\documentclass[smallcondensed]{svjour3}     % onecolumn 
\smartqed
\usepackage{amsmath,amsfonts}
\usepackage{geometry}
\usepackage{pgf}
\usepackage{xspace}
\usepackage{hyperref}
\usepackage{framed}
\usepackage{graphicx}
\usepackage{caption}
\usepackage{subcaption}
\usepackage{booktabs}
\usepackage{enumerate}

\newcommand{\ignore}[1]{}

\newcommand{\fmax}{f_{\text{\rm max}}}
\newcommand{\fmin}{f_{\text{\rm min}}}

\newcommand{\oR}{{\mathbb R}}
\newcommand{\oN}{{\mathbb N}}

\newcommand{\PP}{{ P}}

\newcommand{\R}{{\mathbb R}}
\newcommand{\Z}{{\mathbb Z}}
\newcommand{\Q}{{\mathbb Q}}
\newcommand{\N}{{\mathbb N}}

\let\ve=\mathbf

\renewcommand{\ll}{{\langle}}
\newcommand{\rr}{{\rangle}}

\newcommand{\Conv}{\operatorname{Conv}}

\newcommand{\vol}{{\mathrm{\rm vol}}}
\renewcommand{\ll}{{\langle}}

\renewcommand\d{\mathrm d}

\newcommand{\norm}[1]{\left\lVert#1\right\rVert}
\renewcommand\d{\,\mathrm{d}}

\usepackage{algorithm}
\usepackage{algorithmic}

\begin{document}

\title{Approximating the maximum of a polynomial over a polytope: Handelman decomposition and continuous generating functions
}
%\subtitle{Do you have a subtitle?\\ If so, write it here}

\titlerunning{Handelman and generating functions for polynomial optimization}        % if too long for running head

\author{Jes\'us A. De Loera$^{1,2,5}$ \and Brandon Dutra$^{1,3}$ \and Matthias K\"{o}ppe$^{1,4}$}

\authorrunning{ Jes\'us A. De Loera, Brandon Dutra,   Matthias K\"{o}ppe} % if too long for running head

\institute{
$^1$ Department of Mathematics \\
University of California, Davis \\
One Shields Ave, Davis, CA 95616\\
Tel.: (530) 752-0827\\
$^2$\email{deloera@math.ucdavis.edu} \\
$^3$\email{bedutra@ucdavis.edu} \\
$^4$\email{mkoeppe@math.ucdavis.edu} \\
$^5$Corresponding author
}

\date{Received: date / Accepted: date}
% The correct dates will be entered by the editor

\maketitle

\begin{abstract}
We investigate a way to approximate the maximum of a polynomial $f$ over a polytopal region $P$ through the computation of the integrals $\int_{\PP} f(x)^k\d x$.  We do
a refined analysis of the quality of approximation for the resulting upper and lower bounds. We also propose a new methodology to compute those integrals using a Handelman polynomial decomposition and  continuous multivariate generating functions.  
 

\end{abstract}

%\tableofcontents

%%%%%%%%%%%%%%%%%%%%%%%%%%%%%%%%%%%%%%%%%%%%%%%%%%%%%%%%%%%%%%%
%%%%%%%%%%%%%%%%%%%%%%%%%%%%%%%%%%%%%%%%%%%%%%%%%%%%%%%%%%%%%%%
%%%%%%%%%%%%%%%%%%%%%%%%%%%%%%%%%%%%%%%%%%%%%%%%%%%%%%%%%%%%%%%
%%%%%%%%%%%%%%%%%%%%%%%%%%%%%%%%%%%%%%%%%%%%%%%%%%%%%%%%%%%%%%%

\section{Introduction}
Let $\PP \subset \R^d$ be a $d$-dimensional rational polytope given by the inequality description $\PP := \{ x \in \R^d \mid Ax \leq b\}$ where $A \in \Q^{n \times d}$, and $b \in \Q^n$. We study a new methodology for approximating the global maximum of a polynomial function over the polytope $\PP$. More precisely, we will develop a sequence of lower and upper bounds to the problem

\begin{equation}
\label{equ:continuousGeneralOpt}
\begin{split}
\max & \; f(x) \\
 & x \in \PP,
\end{split}
\end{equation}
where $f(x) \in \Q[x_1, \dots, x_d]$ is a polynomial that is \emph{nonnegative} on $\PP$. We denote this maximum by $\fmax$. We transfer the ideas in \cite{deloera-hemmecke-koeppe-weismantel:intpoly-fixeddim}, where they developed an algorithm for approximating the global maximum of a polynomial function over the \emph{integer points} of a convex polytope. In \cite{deloera-hemmecke-koeppe-weismantel:intpoly-fixeddim}, their bounds involved evaluating the sum $\sum_{x \in \PP \cap \Z^d} f(x)^k$ in polynomial time in $k$ when the dimension $d$ is fixed and the degree of $f$ is bounded. In Section \ref{sec:using-knorms-general} we prove an analogous approximation algorithm when the domain is \emph{continuous} which relies on evaluating the integral $\int_{\PP} f(x)^k\d x$. 

Let us recall prior results on the complexity of maximizing a polynomial over a polytope. When the dimension is allowed to vary, the exact optimization problem is NP-hard, as it includes such problems as the max-cut problem, which is equivalent to quadratic optimization over the hypercube, and the maximum stable set problem, which is equivalent to quadratic optimization over the standard simplex \cite{Hastad99someoptimal}.  Even approximating the optimal objective value of polynomial programming problems is still hard. For instance, Bellare and Rogaway \cite{Bellare:1995} proved that if $P \neq NP$ and $\epsilon \in (0,\frac{1}{3})$ is fixed, there does not exist an algorithm $\mathcal{A}$ for continuous quadratic programming over polytopes that produces an approximation $f_\mathcal{A}$ in polynomial time where $|f_\mathcal{A} - \fmax| < \epsilon |\fmax - \fmin|$. For a survey on the complexity for exactly or approximately optimizing a polynomial on $\PP$, see \cite{koeppeComplexity2012}. 

When the dimension is fixed, there are many polynomial-time
approximation algorithms for polynomial programming. For instance
Parrilo \cite{parrilo2003SDP}, building on work by Lasserre
\cite{Lasserre01globaloptimization} and Nesterov
\cite{nesterov2000sqFunctinalSystems}, developed such methods using
sums of squares optimization. These methods are further developed in
\cite{anjos2011handbook,marshall2008positive}. In
\cite{deKlerk2006210} and
\cite{deloera-hemmecke-koeppe-weismantel:mixedintpoly-fixeddim-fullpaper},
polynomial time approximation algorithms were developed that
discretized the domain. Handelman's theorem (which we discuss below) has  been
used to optimize a polynomial in
\cite{laurentsurvey,monique2014,ParriloSturmfels2003,sankaranarayanan2013lyapunov}. Other
methods for polynomial optimization have been developed in
\cite{Behrends2015,Claudia2014,Thanh2013,tawarmalani2002convexification,baron}.

As we noted our method here relies on integrating $\int_{\PP} f(x)^k\d x$. 
There are many references in the literature on how to compute $\int_{\PP} f(x)^k\d x$ numerically \cite{CUBPACK} or symbolically \cite{bronstein2005symbolic}. We use efficient algorithms for computing $\int_{\PP} f(x)^k\d x$ that take advantage of fast computations with continuous generating functions. Such methods have been developed in recent years \cite{baldoni-berline-deloera-koeppe-vergne:integration,deloera:software-exact-integration-polynomials}. In \cite{baldoni-berline-deloera-koeppe-vergne:integration}, the focus is on the case when the polytope $\PP$ is a simplex and the integrand is a power of a linear form (e.g., $(x_1 + 2x_2 - x_3)^5$), or a product of linear forms (e.g., ${(x_1 + 2x_2)^2(3x_1 + x_3)^4}$). Then in \cite{deloera:software-exact-integration-polynomials}, the case when $\PP$ is an arbitrary full-dimensional polytope and the integrand is a power of a linear form is developed. To integrate a polynomial $f(x) \in \Q[x_1, \dots, x_d]$, both papers decompose $f(x)$ to a sum of powers of linear forms using a simple formula, which we will review in Section \ref{sec:handelman-generating-function}.

The first main contribution of this paper is Theorem \ref{theorem:continuous-lkuk-bounds}. It is the continuous analog of Theorem 1.1 
in \cite{deloera-hemmecke-koeppe-weismantel:intpoly-fixeddim} for optimizing a polynomial $f$ over the continuous domain $\PP$. 
It can also be seen as the limit case of using a grid-summation method for 
polynomial optimization from \cite{deloera-hemmecke-koeppe-weismantel:mixedintpoly-fixeddim}.  The idea to use moments of $f(x)$ for optimization is a common idea, see for example \cite{barvinok2002estimating,barvinok2007integration}. Here we give a new precise analysis of such ideas.

\begin{theorem}
\label{theorem:continuous-lkuk-bounds}
Let the number of variables $d$ be fixed. Let $f(x)$  be a polynomial of total degree $D$ with rational coefficients, and let $\PP$ be a full-dimensional convex rational polytope defined by linear inequalities in $d$ variables. Assume $f(x)$ is nonnegative over $\PP$. Then there are an increasing sequence of lower bounds $\{L_k\}$ and a decreasing sequence of upper bounds $\{U_k\}$ for $k \geq k_0$ to the optimal value of 
\[ \max  f(x) \text{ such that }  x \in \PP\]
that have the following properties:
\begin{enumerate}
\item For each $k > 0$, the lower bound is given by
\[ L_k := \sqrt[k]{\frac{\int_{\PP} f(x)^k \d x}{ \vol(\PP)}}. \] 

\item Let $M$ be the maximum width of $\PP$ along the $d$ coordinate directions, $\gamma_k := \frac{d}{d+k}$, and let $\mathcal{L}$ be a Lipschitz constant of $f$ satisfying
\[ |f(x) - f(y)| \leq \mathcal{L} \norm{x-y}_\infty \text{ for } x,y \in \PP.\]
Then for all $k$ such that $k \geq k_0$ where $k_0 = \max\{1, d(\frac{\fmax}{M\mathcal{L}} -1)\}$, the upper bound is given by

\[U_k := \left( \frac{\int_\PP f(x)^k \d x}{\vol(\PP)} \right)^{1/(d+k)} \cdot (M\mathcal{L})^{\gamma_k} \frac{1}{\gamma_k^{\gamma_k}} \frac{1}{(1-\gamma_k)^{1 - \gamma_k}}.
\]

\item If $M$ and $\mathcal{L}$ are rational, then $L_k^k$ and $U_k^{d+k}$ can be computed exactly as a rational number in polynomial time, for fixed $d$, in unary encoding of $D$ and $k$ and binary encoding of $\PP$ and $f(x)$.

%\item Let $\mathcal{O}_k$ be an oracle for computing $\int_\PP f(x)^k \d x$ as an exact rational number that, for fixed $k$, runs in polynomial time in the unary encoding of $D$ and the binary encoding of $\PP$ and $f(x)$. Let $\eta$ be the number  of precision bits desired in computing the nonnegative real $(d+k)^{th}$ root of a number. Then for fixed $k$, the bounds $L_k$ and $U_k$ can be computed as floating point numbers in polynomial time in the unary encoding of $D$ and the binary encoding of $\PP$ and $f(x)$, and in linear time in $\log \eta$. 

\item Let $\epsilon > 0$, and $U$ be an arbitrary initial upper bound for $\fmax$, then there is a $k \geq k_0$ such that \begin{enumerate}[(i)] \item $U_k - L_k \leq \epsilon\fmax$ where $k$ depends polynomially on $1/\epsilon$, linearly on $U/M\mathcal{L}$, and logarithmically on $M \mathcal{L}/U$, and \item for this choice of $k$, $L_k$ is a $(1-\epsilon)$--approximation to the optimal value $\fmax$ and $U_k$ is a $(1 + \epsilon)$--approximation to $\fmax$.\end{enumerate}
\end{enumerate}
\end{theorem}

We note that Theorem \ref{theorem:continuous-lkuk-bounds} involves integrating the polynomial $f(x)$. Integral kernel or moment methods produce an approximation to $\fmax$ via computing $\int f(x)\d \mu_k$ where the measure $\mu_k$ usually involves a special subset of  nonnegative functions, like sum-of-squares polynomials. Such  methods have been developed in \cite{deKlerk2015,lasserre2009momentsBook,Lasserre2000929,Lasserre01globaloptimization,lasserre2002semidefinite,lasserre2011NewLook}. Our method is 
different as our measure is always the standard Lebesgue measure, while our integrand is $f(x)^k$.  The proof will be presented in Section \ref{sec:using-knorms-general}.
An example of Theorem \ref{theorem:continuous-lkuk-bounds} and a comparison with methods from \cite{deloera-hemmecke-koeppe-weismantel:mixedintpoly-fixeddim-fullpaper} will be presented in Section \ref{example+comparison}.

As we observed before, special decompositions of polynomials are fundamental for fast integration using generating functions. 
In the present paper, we propose an alternative decomposition of the polynomial $f(x)$ into a sum of products of affine linear functions 
(e.g., ${(x_1 + 2x_2 + 3)^2(3x_1 + x_3 - 1)^4}$) using a Handelman decomposition:

\begin{theorem}[Handelman \cite{handelman1988}]\label{theoHan}
Assume that $g_1,\dots,g_n\in\oR[x_1, \dots, x_d]$ are linear polynomials that define a polytope 
\begin{equation}\label{eqK}
\PP=\{x\in\R^d:g_1(x)\ge0,\dots,g_n(x)\ge0\}
\end{equation}
which is compact and has a non-empty interior. Then any polynomial $f\in\oR[x_1, \dots, x_d]$ strictly positive on $\PP$ can be written as
$f(x) = \sum_{\alpha\in\oN^n} c_\alpha g_1^{\alpha_1}\cdots g_n^{\alpha_n}$
for some nonnegative scalars $ c_{\alpha}$.
\end{theorem}

Note that the polynomials $g_i(x)$ correspond to the rows in the constraint matrix $b - Ax \geq 0$ of the polytope $\PP$. In the case of the hypercube $\PP=[0, 1]^d$, this result was shown  earlier by Krivine \cite{Krivine1964}. See \cite{Castle20091285,deKlerk2015,monique2014} for a nice introduction to  Handelman decompositions. A Handelman decomposition is only guaranteed to exist if the polynomial is strictly greater than zero on $\PP$, and the required degree of a Handelman decomposition can grow as the minimum of the polynomial approaches zero. See \cite{powers2001new,sankaranarayanan2013lyapunov} and the references therein.  

Traditionally  Handelman decompositions were used directly (without integration) for optimization. Here
we propose using Handelman decompositions as an alternative way to decompose the input polynomials for integration. This time the decomposition is a sum of products of affine linear forms.
The difference between products of linear forms and products of affine functions is that the latter is allowed to have constant terms in the factors. The Handelman decomposition method has advantages over decomposing $f(x)$ into a sum of powers of linear forms as we will show in Section \ref{sec:handelman-example}.  Our key engine for calculating a sparse Handelman decomposition
is the solution of linear programming problems which are quite efficient and, under  some conditions, are polynomial time computable. We demonstrate the practical potential of this method
to recover short (or sparse) sums through experiments. See Section \ref{sec:handelman-generating-function}.

Our interest in Handelman decompositions of a polynomial was for integration.  We  extended the ideas in \cite{baldoni-berline-deloera-koeppe-vergne:integration,deloera:software-exact-integration-polynomials} to the case when the integrand is a Handelman summand. For this we developed a new generating function method for integrating a product of affine functions over a polytope in Section \ref{sec:gen-fun-for-integration}.

Our second main theorem assures that, when the dimension and the number of inequalities defining the polytope are both constant, one can recover efficiently a Handelman decomposition of some translation $f(x)+s$ of the polynomial $f(x)$. The translation $s$ is necessary when $f(x)$ is not nonnegative on $\PP$.  However, adding an arbitrary large $s$ to $f(x)$ is not ideal as the bounds $L_k$ and $U_k$ from Theorem \ref{theorem:continuous-lkuk-bounds} guarantee that $U_k - L_k \leq \epsilon (\fmax +s)$. 
We therefore use Handelman decompositions as a practical heuristic for finding a sparse decomposition and a small shift so that $f(x)+s$ is nonnegative on $\PP$.

\begin{theorem}
\label{theorem:handelman-good-s-on-simplex}
Fix the dimension $d$ and the number $n$ of facets of the polytope $\PP$. Let $D$ be the degree of $f$. The polynomial $f$ could take negative values on $\PP$. 
\begin{enumerate}
\item There is an algorithm (Algorithm \ref{alg:knorm-handelman}) that runs in time polynomial in the unary encoding size of $k$, $D$, and the binary encoding  size of $f$ and $\PP$,   that  computes a number $s$ with $f(x)+s \geq 0$ on $\PP$, determines a Handelman decomposition of $(f(x)+s)^k$, and evaluates the integral $\int_\PP (f(x)+s)^k \d x$. 

\item Moreover, when $\PP$ is a full dimensional simplex and $\epsilon > 0$ is a small number, an $s$ can be computed in polynomial time in the unary encoding size of $D$ and the binary encoding size of $f(x)$ and $\PP$  such that Theorem \ref{theorem:continuous-lkuk-bounds} produces bounds on $f(x)+s$ with $U_k - L_k \leq \epsilon (\fmax - \fmin)$ where $k$ depends polynomially on $1/\epsilon$ and $D^D \binom{2D-1}{D}$. 
\end{enumerate}
\end{theorem}

The proof of this theorem will be presented in the last section of the paper, Section \ref{proof-thm-last}.

%%%%%%%%%%%%%%%%%%%%%%%%%%%%%%%%%%%%%%%%%%%%%%%%%%%%%%%%%%%%%%%
%%%%%%%%%%%%%%%%%%%%%%%%%%%%%%%%%%%%%%%%%%%%%%%%%%%%%%%%%%%%%%%
%%%%%%%%%%%%%%%%%%%%%%%%%%%%%%%%%%%%%%%%%%%%%%%%%%%%%%%%%%%%%%%
%%%%%%%%%%%%%%%%%%%%%%%%%%%%%%%%%%%%%%%%%%%%%%%%%%%%%%%%%%%%%%%

\section{Using integration for optimization over a general polytope}
\label{sec:using-knorms-general}

This section is devoted to the proof of Theorem \ref{theorem:continuous-lkuk-bounds}.

\begin{proof}\textit{of Theorem \ref{theorem:continuous-lkuk-bounds}, part 1:}
This follows from the fact that $\int_\PP f(x)^k \d x \leq \int_\PP \fmax^k \d x.$ \qed
\end{proof}

\begin{proof}\textit{of Theorem \ref{theorem:continuous-lkuk-bounds}, part 2:} We take inspiration from the standard proof that $\lim_{p \rightarrow \infty} \norm{f}_p = \norm{f}_\infty$. Let  
$\gamma > 0$
, $\PP^{\gamma}:= \{ x \in \PP \mid f(x) \geq (1-\gamma)\fmax\}$ and $[\PP^{\gamma}]$ denote the characteristic function of $\PP^{\gamma}$. Then we have 
\[(1-\gamma)\fmax \cdot \vol(\PP^{\gamma}) \leq \int_{\PP^{\gamma}} f(x) \d x = \int_{\PP} f(x)[\PP^{\gamma}] \d x \leq \left( \int_\PP f(x)^k \d x \right)^{1/k} \vol(\PP^{\gamma})^{1/q}\]
where the last inequality comes from H\"older's inequality with $\frac{1}{k} + \frac{1}{q} = 1$. Rearranging the first and last expressions results in 
\[ \left( \int_\PP f(x)^k \d x \right)^{1/k} \geq (1-\gamma)\fmax \cdot \vol(\PP^{\gamma})^{1/k}.\]

We now seek a lower bound for $\vol(\PP^{\gamma})$. Because $f$ is a polynomial and $\PP$ is compact, $f$ has a Lipschitz constant $\mathcal{L}$ ensuring $|f(x) - f(y)| \leq \mathcal{L} \norm{x-y}_\infty$ for $x,y \in \PP$. Let $x^*$ denote an optimal point where $f(x^*) = \fmax$ and let \[B_\infty(x^*,r) := \{x \in \R^d : \norm{x-x^*}_\infty \leq r \}\] be a closed  ball of radius $r$ centered at $x^*$. Then for all $x$ in the ball $B_\infty(x^*, \gamma \fmax / \mathcal{L})$, $f(x) \geq (1-\gamma)\fmax$. Therefore,

\[\vol(P^{\gamma}) = \vol\{x \in P \mid f(x) \geq (1-\gamma)\fmax\} \geq \vol\{ P \cap B_\infty(x^*,\gamma \fmax / \mathcal{L}) \}. \]

\begin{figure}
        \centering
        \begin{subfigure}[b]{0.3\textwidth}
                \includegraphics{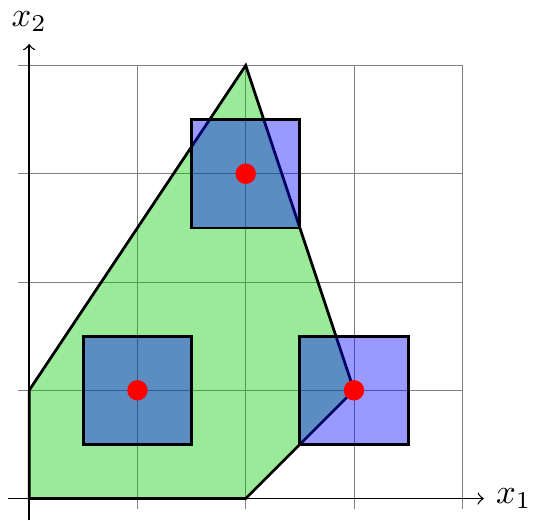}
                \caption{}
                \label{fig:pentagon-without}
        \end{subfigure}%
        \hspace{1in} %add desired spacing between images, e. g. ~, \quad, \qquad, \hfill etc.
          %(or a blank line to force the subfigure onto a new line)
        \begin{subfigure}[b]{0.3\textwidth}
                \includegraphics{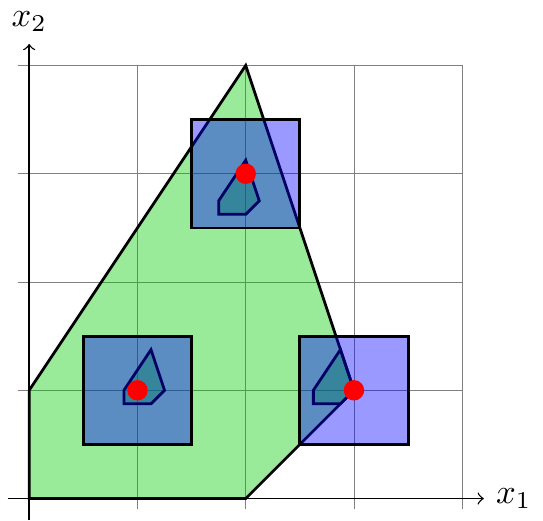}
                \caption{}
                \label{fig:pentagon-with}
        \end{subfigure}
\caption{(\protect\subref{fig:pentagon-without}) Pentagon with three possible points $x^*$ that maximize $f(x)$ along with the $\ell_\infty$ balls $B_\infty(x^*, 1/2)$. It might not be true that the ball is contained in the polytope. (\protect\subref{fig:pentagon-with}) However, a scaled pentagon can be contained in the ball and the original polytope.}
\label{fig:pentagons-scaled-balls}
\end{figure}

Notice that the ball $B(x^*,\gamma \fmax / \mathcal{L})$ may not be contained in $\PP$, see Figure \ref{fig:pentagon-without}. We need to lower bound the fraction of the ball that intersects $\PP$ where $x^*$ could be any point in $\PP$. To do this, we will show how a scaled $\PP$ can be contained in $\PP^{\gamma}$ as in Figure \ref{fig:pentagon-with}. Consider the invertible affine transformation $\Delta : \R^d \rightarrow \R^d$ given by $\Delta x := \delta(x-x^*) + x^*$ where $\delta > 0$ is some unknown scaling term. Notice that this function is simply a scaling about the point $x^*$ that fixes the point $x^*$. We need to pick $\delta$ such that $\Delta \PP \subseteq \PP \cap B_\infty(x^*,\gamma \fmax / \mathcal{L})$. We claim that choosing $\delta = \min\{1,\frac{\gamma \fmax}{M\mathcal{L}}\}$ suffices where $M$ is the maximum width of $\PP$ along the $d$ coordinate axes: $M := \max_{x, y \in \PP} \norm{x - y}_\infty$. First, since $\delta < \frac{\gamma \fmax}{M\mathcal{L}}$,
\[\norm{\Delta x - x^*}_\infty = \norm{\delta (x-x^*) + x^* -x^*}_\infty = \delta \norm{ x- x^*}_\infty \leq \frac{\gamma \fmax}{\mathcal{L}}.\]
Second, since $\delta \in (0,1)$ and $\PP$ is convex, 
\[\Delta x = \delta x + (1 - \delta) x^* \in \PP.\] 

%\begin{itemize}
%\item If $\Delta \PP \subseteq B_\infty(x^*,\gamma \fmax / \mathcal{L})$, then for $x \in \PP$
%\begin{align*}
%\norm{\Delta x - x^*}_\infty &= \norm{\delta (x-x^*) + x^* -x^*}_\infty \\
%&= \delta \norm{ x- x^*}_\infty \\
%&\leq \frac{\gamma \fmax}{\mathcal{L}}.  \\
%\end{align*}
%
%Let $M$ be the maximum width of $\PP$ along the $d$ coordinate axes, then 
% $0 < \delta \leq \frac{\gamma \fmax}{M\mathcal{L}}$ is a sufficient condition for  $\Delta \PP \subseteq B_\infty(x^*, \gamma \fmax/\mathcal{L})$.
%
%\item Let $x \in \PP$, and write $x$ and $x^*$ as a convex combination of the vertices of $\PP$. That is, let $x=\sum_{i=1}^N \alpha_i v_i$ and $x^*=\sum_{i=1}^N \beta_i v_i$ where $\sum_{i=1}^N \alpha_i = 1$, $\sum_{i=1}^N \beta_i = 1$, $\alpha_i \geq 0$, and $\beta_i \geq 0$. Then 
%
%\begin{align*}
%\Delta x &= \delta \left(\sum_{i=1}^N \alpha_i v_i- \sum_{i=1}^N \beta_i v_i\right) + \sum_{i=1}^N \beta_i v_i \\
%&= \sum_{i=1}^N (\delta \alpha_i + \beta_i - \delta \beta_i) v_i \\
%\end{align*}
%
%No matter what $\delta$ is, $\sum_{i=1}^N (\delta \alpha_i + \beta_i - \delta \beta_i) = 1$. Forcing $\delta \alpha_i +  (1 - \delta)\beta_i \geq 0$ for each $i$ as $x$ and $x^*$ varies over the polytope $\PP$ is equivalent to $\delta \leq 1$. 
%
%Hence, $0 < \delta \leq 1$ a sufficient condition for $\Delta \PP \subseteq \PP$. 
%
%\end{itemize}
Therefore, $\vol(P \cap  B_\infty(x^*,\gamma \fmax / \mathcal{L})) \geq \vol(\Delta \PP) = \left(\min\{1,\frac{\gamma \fmax}{M\mathcal{L}}\}\right)^d \vol(\PP),$ and finally, 

\begin{equation}
 \left( \int_\PP f(x)^k \d x \right)^{1/k} \geq (1-\gamma)\fmax \cdot \vol(\PP)^{1/k} \left(\min\{1,\frac{\gamma \fmax}{M\mathcal{L}}\}\right)^{d/k}. \label{ineq:lpnormbound}
\end{equation}

As the above inequality is true for all $0 < \gamma \leq 1$, we want to pick $\gamma$ to maximize the function $\phi(\gamma) := (1-\gamma) \left(\min\{1,\frac{\gamma \fmax}{M\mathcal{L}}\}\right)^{d/k}$. The maximum of $(1-\gamma) \left(\frac{\gamma \fmax}{M\mathcal{L}}\right)^{d/k}$ occurs at $\gamma = \frac{d}{d+k}$. Hence the maximum of $\phi(\gamma)$ occurs at $\gamma = \frac{d}{d+k}$ if this is less than $\frac{M\mathcal{L}}{\fmax}.$ Otherwise if $\gamma \geq \frac{M\mathcal{L}}{\fmax}$, the maximum of $\phi(\gamma)$ is at $\gamma = \frac{M\mathcal{L}}{\fmax}$. Therefore  the maximum of $\phi(\gamma)$ occurs at $\gamma =  \min\{\frac{d}{d+k}, \frac{M\mathcal{L}}{\fmax}\}$. Enforcing $\frac{d}{d+k} \leq \frac{M\mathcal{L}}{\fmax}$ is equivalent to $d(\frac{\fmax}{M\mathcal{L}} - 1) \leq k$. Let $\gamma_k := \frac{d}{d+k}$, then with this value restriction on $k$, we have $\gamma = \gamma_k$ and solving for $\fmax$ in Inequality \eqref{ineq:lpnormbound} yields the desired formula for $U_k$.
 \qed
\end{proof}

An important input to our approximation bounds is a Lipschitz constant $\mathcal{L}$ satisfying $|f(x) - f(y)| \leq \mathcal{L} \norm{x-y}_\infty$ for all $x, y \in \PP$. One natural way to compute this constant is by maximizing $\norm{\nabla f(x)}_1$ on $\PP$. However this is a difficult problem in general. Instead, one can compute a potentially larger constant by following the next lemma.

\begin{lemma}[Rewording of Lemma 11 in \cite{deloera-hemmecke-koeppe-weismantel:mixedintpoly-fixeddim-fullpaper}] 
\label{lemma:lemma_11}
Let $f$ be a polynomial in $d$ variables with maximum total degree $D$. Let $c$ denote the largest absolute value of a coefficient of $f$, and let $r$ denote the number of monomials of $f$. Then there exists a Lipschitz constant $\mathcal{L}$ such that $|f(x) - f(y)| \leq \mathcal{L} \norm{x-y}_\infty$ for all $|x_i|, |y_i| \leq \tilde{M}$. The constant $\mathcal{L}$ is $crD\tilde{M}^{D-1}$. 
\end{lemma}

\begin{proof}\textit{of Theorem \ref{theorem:continuous-lkuk-bounds}, part 3:}

We first expand the polynomial $f(x)^k$ with respect to the standard monomial basis. Let $g(x) := f(x)^k$. The authors of \cite{baldoni-berline-deloera-koeppe-vergne:integration,deloera:software-exact-integration-polynomials} have developed algorithms for computing $\int_\PP g(x) \d x$ as an exact rational number that runs in polynomial time in the unary encoding of the degree of $g(x)$, and the binary encoding of $\PP$ and $g(x)$. 

The maximum width of $\PP$ along the $d$ coordinate directions, $M$, can be computed by solving $2d$ linear programs in the form $\min_{x \in \PP} x_i$ and $\max_{x \in \PP} x_i$ in polynomial time. 

Lemma \ref{lemma:lemma_11} produces a valid and easily computable Lipschitz constant for $f(x)$ on a box containing $\PP$. The $\tilde{M}$ in Lemma \ref{lemma:lemma_11} can be found from the same linear programs used to find $M$.
\qed
\end{proof}

Computing $L_k$ from $L_k^k$  requires approximating the real nonnegative $k$-th root with a lower bound, and similarly for $U_k$. This is the only source of numerical approximation that enters our method. Approximating these $k$-th or $(d+k)$-th roots with $\eta$ bits of precision can be done with $O(\log \eta)$ iterations of a Newton-Raphson type method, see \cite{chen1989fast}. We also remark that Sections \ref{sec:handelman-generating-function} and \ref{sec:gen-fun-for-integration} describe another polynomial time algorithm for computing $\int_\PP f(x)^k \d x$ that is similar to the methods in \cite{baldoni-berline-deloera-koeppe-vergne:integration,deloera:software-exact-integration-polynomials}.

\begin{lemma}
\label{lemma:loge}
$\frac{1}{\log(1+\epsilon)} \leq 1 + \frac{1}{\epsilon}$ for $0 < \epsilon \leq 1.$
\end{lemma}
\begin{proof}
It is enough to show that $\phi(\epsilon):= (1 + \frac{1}{\epsilon})\log(1+\epsilon) \geq 1$ for $0 < \epsilon \leq 1$. This follows because $\phi(\epsilon)$ is increasing on $0 < \epsilon \leq 1$ and $\lim_{\epsilon \rightarrow 0} \phi(\epsilon) = 1$ from L'H\^{o}pital's rule.  \qed
\end{proof}

%\begin{lemma}
%\label{lemma:logz}
%$\log(z)/z \leq 1/\sqrt{z}$ for $z \geq 1$.
%\end{lemma}
%\begin{proof}
%We show that $\phi(z):= \log(z) / \sqrt{z} \leq 1$ for all $z \geq 1$.
%Notice that $\phi(1) = 0 < 1$. Next we maximize $\phi$. Solving for $\frac{d}{dz}\phi(z) =0$ gives $z = e^2$. Then $\phi(e^2) = \frac{2}{\sqrt{e^2}} < 1$. Finally, $\lim_{z \rightarrow \infty} \phi(z) = 0$. 
%\end{proof}

\begin{lemma}
\label{lemma:logdelta}
For every $\delta > 0$, there is $c_\delta > 0$ such that \[\log(z) \leq c_\delta z^{\frac{\delta}{1+\delta}}\] for all $z > 0$.  In particular, if $\delta = 0.1$, then $c_\delta$ can be set to $4.05$.
\end{lemma}
\begin{proof}
Because \[\lim_{z \rightarrow \infty} \frac{z^{\frac{\delta}{1+\delta}}}{\log(z)} = \infty,\]
there exist a $a_1 > 0$ and $z_0 \geq 0$ such that $ \log(z) \leq a_1 z^{\frac{\delta}{1+\delta}}$ for $z \geq z_0$. When $z_0 > 1$, another constant $a_2$ exists such that $\log(z) \leq a_2 z^{\frac{\delta}{1+\delta}}$ for $1 \leq z \leq z_0$, and $c_\delta$ can be assigned $\max\{a_1, a_2\}$. Otherwise $c_\delta = a_1$ is sufficient.

Now let $\delta = 0.1$. We show the minimum of $\phi(z) := 4.05(z)^{1/11} - \log(z)$ is positive. Notice that $\phi(z) \rightarrow \infty$ as $z \rightarrow \infty$ or as $z \rightarrow 0$. The only zero of $\phi'(z)$ occurs at $z^* = \left(\frac{11}{4.05}\right)^{11}$. Finally, $\phi(z^*) > 0$. \qed

\end{proof}

%\begin{lemma}
%\label{lemma:logdelta}
%For every $\delta > 0$, there is a $c_\delta > 0$ such that \[\frac{\log(z+1)}{z+1} \leq c_\delta (z+1)^{- \frac{1}{1+\delta}} \] for all $z > 0$.
%
%In particular, if $\delta = 0.1$, then $c_\delta$ can be set to $4.05$.
%\end{lemma}
%\begin{proof}
%Because \[\lim_{z \rightarrow \infty} \frac{(z+1)^{1 - \frac{1}{1+\delta}}}{\log(1+z)} = \infty,\]
%there exist a $c > 0$ and $z_0 \geq 0$ such that $ \log(z+1) \leq c(z+1)^{1 - \frac{1}{1+\delta}}$ for $z \geq z_0$. As $(z+1)^{1 - \frac{1}{1+\delta}}$ is positive for $z \geq 0$, there exist a sufficiently large constant, $c_\delta$, such that $ \log(z+1) \leq c_\delta(z+1)^{1 - \frac{1}{1+\delta}}$ for $z \geq 0$.
%
%Now let $\delta = 0.1$. We show the minimum of $\phi(z) := 4.05(z+1)^{1/11} - \log(z+1)$ is positive. Notice that $\phi(0) > 0$, and $\lim_{z \rightarrow \infty} \phi(z) = \infty$. The only zero of $\phi'(z)$ occurs at $z^* = \left(\frac{11}{4.05}\right)^{11} -1$. Finally, $\phi(z^*) > 0$. \qed
%
%\end{proof}

\begin{proof}\textit{of Theorem \ref{theorem:continuous-lkuk-bounds}, part 4:}

First let $k \geq d(\frac{U}{M\mathcal{L}} -1)$. This ensures the formula for $U_k$ holds because \[k \geq d\left(\frac{U}{M\mathcal{L}}-1\right) \geq d\left(\frac{\fmax}{M\mathcal{L}}-1\right).\] In terms of an algorithmic implementation, other upper bounds for $d\left(\frac{\fmax}{M\mathcal{L}}-1\right)$ can be used; for instance, because $\fmax - \fmin \leq M\mathcal{L}$, $d\left(\frac{\fmax}{M\mathcal{L}}-1\right) \leq \frac{d\fmin}{M\mathcal{L}} \leq \frac{df(x_0)}{M\mathcal{L}}$ where $x_0$ is any point in the domain. 

Note that $\frac{1}{d+k}-\frac{1}{k} = \frac{1}{k} (-\gamma_k)$, then with $U_k$ and $L_k$ as defined above, we have the chain of inequalities

\begin{align*}
U_k - L_k  &= \left(  \left(\frac{\int_\PP f^k \d x}{\vol(\PP)}\right)^{\frac{1}{d+k}} (M\mathcal{L})^{\gamma_k}\frac{1}{\gamma_k^{\gamma_k}} \frac{1}{(1-\gamma_k)^{1-\gamma_k}} \right) - \left(\frac{\int_\PP f^k \d x}{\vol(\PP)}\right)^{1/k} \\
& = \left(\frac{\int_\PP f^k \d x}{\vol(\PP)}\right)^{1/k} \left(  \left(\frac{\int_\PP f^k \d x}{\vol(\PP)}\right)^{\frac{1}{k}(-\gamma_k)} (M\mathcal{L})^{\gamma_k}\frac{1}{\gamma_k^{\gamma_k}} \frac{1}{(1-\gamma_k)^{1-\gamma_k}}   - 1\right)  \\
& \leq \fmax \left(  \left(\frac{\int_\PP f^k \d x}{\vol(\PP)}\right)^{\frac{1}{k}(-\gamma_k)} (M\mathcal{L})^{\gamma_k}\frac{1}{\gamma_k^{\gamma_k}} \frac{1}{(1-\gamma_k)^{1-\gamma_k}}   - 1\right)  \\
&\leq \fmax \left(  U^{-\gamma_k} (M\mathcal{L})^{\gamma_k}\frac{1}{\gamma_k^{\gamma_k}} \frac{1}{(1-\gamma_k)^{1-\gamma_k}}   - 1\right) \\
& = \fmax \left(  \left(\frac{M \mathcal{L}}{U}\right)^{\gamma_k}\frac{1}{\gamma_k^{\gamma_k}} \frac{1}{(1-\gamma_k)^{1-\gamma_k}}   - 1\right).
\end{align*}

Let $\epsilon > 0$. We now show how to pick $k$ large enough so that $\left(\frac{M \mathcal{L}}{U}\right)^{\gamma_k}\frac{1}{\gamma_k^{\gamma_k}} \frac{1}{(1-\gamma_k)^{1-\gamma_k}}  \leq 1+\epsilon$.  We divide this into showing three inequalities:

\begin{align}
\frac{1}{(1-\gamma_k)^{1-\gamma_k}}  \leq (1+\epsilon)^{\frac{1}{3}} \label{ineq:1MinusGamma} \\
\left(\frac{M \mathcal{L}}{U}\right)^{\gamma_k} \leq (1+\epsilon)^{\frac{1}{3}} \label{ineq:uml}\\ 
\frac{1}{\gamma_k^{\gamma_k}} \leq (1+\epsilon)^{\frac{1}{3}} \label{ineq:gammagamma}.
\end{align}

If $k \geq \frac{d}{(\epsilon + 1)^{1/3} -1} $, then after rearranging this we have
\[ (\epsilon + 1)^{1/3} \geq  \frac{d}{k} + 1 = \frac{d+k}{k} = \frac{1}{1 - \gamma_k} \geq \frac{1}{(1 - \gamma_k)^{1 - \gamma_k}}, \]
and so Inequality \eqref{ineq:1MinusGamma} follows.

If $k \geq 3d \log\left(\frac{M \mathcal{L}}{U}\right)\left( 1 + \frac{1}{\epsilon}\right)$, then using  Lemma \ref{lemma:loge} and the fact that $\log(\epsilon+1) > 0$, we have
\[ k \geq 3d \log\left(\frac{M \mathcal{L}}{U}\right)\left( 1 + \frac{1}{\epsilon}\right) \geq \frac{3d \log\left(\frac{M \mathcal{L}}{U}\right)}{\log(\epsilon + 1)} \geq \frac{d[ \log\left(\frac{M \mathcal{L}}{U}\right) - \log(\epsilon + 1)/3]}{\log(\epsilon + 1)/3}. \]
Further rearranging this yields $ \frac{\log(\epsilon + 1)}{3} \geq \frac{d}{d+k}\log\left(\frac{M \mathcal{L}}{U}\right),$ and Inequality \eqref{ineq:uml} now follows.

Notice that Lemma \ref{lemma:logdelta} implies 
\begin{equation}
\log(\gamma_k^{-\gamma_k}) = \gamma_k \log(\frac{1}{\gamma_k}) \leq \gamma_k \frac{c_\delta}{\gamma_k^{\frac{\delta}{1+\delta}}} = c_\delta \gamma_k^{\frac{1}{1+\delta}}. \label{ineq:applylemmalogdelta}
\end{equation}
Inequality \eqref{ineq:applylemmalogdelta} shows that it is sufficient to choose $\gamma_k$ such that \begin{equation}
c_\delta \gamma_k^{\frac{1}{1+\delta}} \leq \frac{1}{3}(\log(1+\epsilon)) \label{ineq:solveforgamma}
\end{equation} for Inequality \eqref{ineq:gammagamma} to  hold. Solving for $k$ in Inequality \eqref{ineq:solveforgamma} produces  $k \geq d\left( (3c_\delta)^{1+\delta}  \frac{1}{\log(1+\epsilon)^{1+\delta}} - 1\right)$. Applying Lemma \ref{lemma:loge} shows that 
$k \geq d\left( (3c_\delta)^{1+\delta} \left( 1 + \frac{1}{\epsilon}\right)^{1+\delta} - 1\right)$ is also sufficient for Inequality \eqref{ineq:gammagamma} to hold.

Therefore let
\begin{equation}
\label{eq:what-k-should-be-given-e}
 k = \left \lceil{\max\left\{d\left(\frac{U}{M\mathcal{L}}-1\right), \frac{d}{(\epsilon + 1)^{1/3} -1}, 3d \log\left(\frac{M \mathcal{L}}{U}\right)\left( 1 + \frac{1}{\epsilon}\right), O\left(\left(1 + \frac{1}{\epsilon}\right)^{1+\delta}\right)\right\}}\right \rceil,
\end{equation} 
 where the last term is understood to mean for every $\delta > 0$, there exist a $c_\delta > 0$ so that $k$ should be larger than $d\left( (3c_\delta)^{1+\delta} \left( 1 + \frac{1}{\epsilon}\right)^{1+\delta} - 1\right)$. In particular, if $\delta = 0.1$, then $c_\delta$ could be $4.05$. Notice that $\frac{1}{(\epsilon + 1)^{1/3} -1} = \frac{1}{\epsilon} + 1 + O(\epsilon)$, so $k$ is bounded by a polynomial in $1/\epsilon$. For this choice of $k$, Inequalities \eqref{ineq:1MinusGamma}, \eqref{ineq:uml}, \eqref{ineq:gammagamma} are all true and $U_k - L_k \leq \fmax \cdot \epsilon$.

$L_k$ is a $(1-\epsilon)$--approximation to $\fmax$ because
\[ \fmax \leq U_k = L_k + (U_k - L_k) \leq L_k +\epsilon \fmax \]
and $U_k$ is a $(1+\epsilon)$--approximation to $\fmax$ because

\[ U_k - \epsilon \fmax \leq U_k + (L_k - U_k) = L_k \leq \fmax.\]
\qed
\end{proof}

Before closing this section, we make three remarks. First, Theorem \ref{theorem:continuous-lkuk-bounds} could also be applied for global \emph{unconstrained} maximization. If a polynomial has a global maximum, then there is a $\mathcal{M} > 0$ such that a global maximizer is contained in the box $\PP := [-\mathcal{M}, \mathcal{M}]^d$.  In fact, the theory of quantifier elimination gives worst-case bounds for representatives of each connected component of the set of global maximizers. These bounds $\mathcal{M}$ are effectively computable from the encoding sizes of the input, and when the dimension $d$ is fixed, $\log \mathcal{M}$ is bounded polynomially;  see Proposition 1.3 in \cite{renegar1992computational} or \cite{vorobjov1984bounds}. In actual applications, tighter bounds may be available.

Second, Theorem \ref{theorem:continuous-lkuk-bounds} gives an approximation to $\max f(x)$ on $\PP$, it does not find a point in the domain that has this value. However, Theorem \ref{theorem:continuous-lkuk-bounds} can also be used to approximate the point in $\PP$ that gives $\max f(x)$ by using a standard bisection method.

Third, the bounds $L_k$ and $U_k$ are also valid for when $\PP$ is a full dimensional convex set, not just a polytope. The difficulty in using Theorem \ref{theorem:continuous-lkuk-bounds} on a convex set is the need to integrate a polynomial over it. 

%%%%%%%%%%%%%%%%%%%%%%%%%%%%%%%%%%%%%%%%%%%%%%%%%%%%%%%%%%%%%%%
%%%%%%%%%%%%%%%%%%%%%%%%%%%%%%%%%%%%%%%%%%%%%%%%%%%%%%%%%%%%%%%
%%%%%%%%%%%%%%%%%%%%%%%%%%%%%%%%%%%%%%%%%%%%%%%%%%%%%%%%%%%%%%%
%%%%%%%%%%%%%%%%%%%%%%%%%%%%%%%%%%%%%%%%%%%%%%%%%%%%%%%%%%%%%%%

\subsection{Examples and comparisons with other approximations} \label{example+comparison}

% OLD BOX EXAMPLE
%We now illustrate our bounds on an example. Let $f(x) = x_1^2x_2 - x_1x_2$ with $x_1 \in [1,3]$ and $x_2\in [1,3]$. Note that $f(x)$ is nonnegative on its domain $\PP = [1,3]^2$. For the Lipschitz constant, we could easily maximize $\norm{\nabla f(x)}_1$ on $\PP$ which for this problem is $21$. However let us use the  bound produced by Lemma \ref{lemma:lemma_11} which gives $\mathcal{L} = 33$. Using the definition of $L_k$ and $U_k$ from Theorem \ref{theorem:continuous-lkuk-bounds}, we have $L_k \leq \fmax \leq U_k$
%where
%$d=2, M = 2$, $\gamma = \frac{2}{2+k}$, and $\mathcal{L} = 33$. Then for different values of $k$, the bounds are:
%\begin{center}
%\begin{tabular}{ c  c c }
%
%  $k$ & $L_k$ & $U_k$\\
%  \hline
%  10 & 11.07 & 23.40  \\
%  20 & 13.22 & 20.75  \\
%  30 & 14.27 & 19.84  \\
%  40 & 14.91 & 19.38  \\    
%\end{tabular}
%\end{center}
%
%Next we want to apply Lemma \ref{lemma:epsilon-polynom-in-k} when $\epsilon=0.1$. Using $U = 19.38$, we see that $k$ has to be at least 
%$ k = \lceil \max\{-1.4, 62.0, 472.2,  434.1\}\rceil$ to guarantee that $U_k - L_k \leq \fmax \cdot \epsilon = 1.8$. However for this example   $k=126$ suffices: $U_{126} - L_{126} = 1.7$.

We now illustrate our bounds on an example. Let $f(x_1,x_2) := -5(x_1^2-2)^2 - 7(x_2^2-2)^2 + 20$ and set the domain $\PP$ to be the triangle with vertices $(1,1)$, $(1,2)$, and $(2,1)$. On this domain, $\fmin = 8$ at $(1,1)$, and $\fmax = 20$ at $(\sqrt{2}, \sqrt{2})$. For the Lipschitz constant, we could maximize $\norm{\nabla f(x)}_1$ on $\PP$ which for this problem is $132$.  The  bound produced by Lemma \ref{lemma:lemma_11} gives $\mathcal{L} = 28\cdot 4 \cdot 4 \cdot 2^3 = 3584$. However, we can apply Lemma \ref{lemma:lemma_11} on each monomial of $f(x)$ and get a smaller Lipschitz constant $\mathcal{L} = 536$, which we use. Using the definition of $L_k$ and $U_k$ from Theorem \ref{theorem:continuous-lkuk-bounds}, we have $L_k \leq \fmax \leq U_k$
where
$d=2, M = 1$, $\gamma_k = \frac{2}{2+k}$, and $\mathcal{L} = 536$. Then for different values of $k$, the bounds for $L_k$ and $U_k$ are shown in the first three columns of Table \ref{table:Lk-Uk_Lkm-summation}. Next we want to apply Theorem \ref{theorem:continuous-lkuk-bounds}, part 4 when $\epsilon=0.1$. Using $U = 27$, $\delta=0.1$, we see that $k$ from Equation \eqref{eq:what-k-should-be-given-e} has to be at least 
$ k = \lceil \max\{-1.9, 62.0, 197.2,  217.1\}\rceil$ to guarantee that $U_k - L_k \leq \fmax \cdot \epsilon = 2.0$. However for this example $k=179$ suffices: $U_{179} - L_{179} <~1.998$.

We also want to compare the bounds in Theorem \ref{theorem:continuous-lkuk-bounds} with the grid-summation approximation ideas in \cite{deloera-hemmecke-koeppe-weismantel:mixedintpoly-fixeddim} restricted to the pure continuous case. Using a grid-summation method has been done before, for instance in \cite{deKlerk2006210}, but there generating functions were not used. Note that \cite{deloera-hemmecke-koeppe-weismantel:mixedintpoly-fixeddim} only gives an algorithm for lower bounding $\fmax$ while we also have the upper bound $U_k$. Let $m \in \N$ and define
\[ S(m) := \sum_{x \in \PP \cap \frac{1}{m} \Z^d} f(x)^k, \text{ and } L_{k,m} := \sqrt[k]{ \frac{S(m)}{|\PP \cap \frac{1}{m} \Z^d|}}. \]
We immediately have that $L_{k,m} \leq \max_{x \in \PP \cap \frac{1}{m} \Z^d} f(x) \leq \fmax$, and so $L_{k,m}$ is a valid lower bound for any $k$ and $m$. Keeping $k$ fixed and taking $m \rightarrow \infty$ reveals that $L_{k,m}$ is an approximation to $L_k$:

\[\lim_{m \rightarrow \infty} L^k_{k,m} = \lim_{m \rightarrow \infty} \frac{S(m)}{m^d} \cdot \frac{m^d}{|\PP \cap \frac{1}{m} \Z^d|} = \int_\PP f(x)^k \d x \cdot \frac{1}{\vol(\PP)} = L_k^k.  \]
Table \ref{table:Lk-Uk_Lkm-summation} shows this convergence for $L_{k,m}$.

\begin{table}
\centering
\caption{Values for $U_k$, $L_k$, and $L_{k,m}$ for $k=10, 20, 30, 40$ and $m=10, 100, 1000$ when $\PP$ is the triangle with vertices $(1,1)$, $(1,2)$, and $(2,1)$.}
\label{table:Lk-Uk_Lkm-summation}
\begin{tabular}{ c  c c c c c }
	\toprule
     &       &       & \multicolumn{3}{c}{$L_{k,m}$} \\  \cmidrule{4-6}
  $k$& $L_k$ & $U_k$ &  $10$ & $100$  & $1000$\\ \midrule
  10 & 17.12 & 47.69 & 16.77 & 17.08  & 17.11  \\
  20 & 17.99 & 33.18 & 17.77 & 17.95  & 17.99  \\
  30 & 18.40 & 28.70 & 18.25 & 18.39  & 18.40  \\
  40 & 18.67 & 26.52 & 18.55 & 18.65  & 18.66  \\  \bottomrule  
\end{tabular}
\end{table}

One problem with using $L_{k,m}$ is that a choice of $m$ has to be made. The choice for $m$ can affect how $L_{k,m}$ converges to $L_k$. To see this, consider a new example where $\PP$ is the one dimensional polytope $[-\frac{1}{4}, \frac{1}{4}]$ and the polynomial is $f(x):= -10x^2 + 2$. Table \ref{table:Lkm-summation-mod-4} shows $L_{1,m}$ for some values of $m$, where the columns are the residue classes of $m$ modulo $4$. The value of $L_{1}$ in this case is $1.791667$. We can see that $L_{1,m}$ is converging to $L_1$ from below in residue classes $0$ and $1$, while it is converging from above in the other classes. Also, the convergence in residue class $2$ is much faster than in the others. This faster convergence can be explained. The Ehrhart quasi-polynomial that counts  $|\PP \cap \frac{1}{m} \Z^d|$ is $\frac{1}{2}m  - \frac{1}{2} (m \bmod 4) +1$. Then \[L_{1,m} = \frac{S(m)}{m} \cdot \frac{m}{\frac{1}{2}m  - \frac{1}{2} (m \bmod 4) +1}.\] Two special things happens when $m \bmod 4 = 2$. First, $\frac{m}{\frac{1}{2}m  - \frac{1}{2} (m \bmod 4) +1} = \frac{1}{\vol(\PP)}$, and then the sum $S(m)/m$ is exactly the Midpoint rule of numerical integration where each interval width is $1/m$. 

%To see this, let $m = 4k + 2$ for some $k\in \Z$. Then $k/(4k+2)$ is the largest lattice point in $[-1/4, 1/4] \cap \Z$. Adding $1/(2m)$ to $k/(4k+2)$ gives $1/4$. Hence, $\PP$ can be partitioned into intervals of length $1/m$ where the lattice points $[-1/4, 1/4] \cap \Z$ are the midpoints. 

\begin{table}
\centering
\caption{$L_{1,m}$ for $f(x)= -10x^2 + 2$ on $\PP = [-\frac{1}{4}, \frac{1}{4}]$  arranged by the residue classes of $m$ modulo $4$ where $m \in \{1, 2, \dots, 15\} \cup \{988, 989, \dots, 999\}$. Bold digits show the correct digits to the limit $1.791667\dots$.}
\label{table:Lkm-summation-mod-4}
\begin{tabular}{@{}rlrlrlrl@{}} \toprule
%\multicolumn{8}{c}{$L_{1,m}$ for $m \in \{1, 2, \dots, 15\} \cup \{988, 989, \dots, 999\}$}      \\ \midrule
Residue  0             & $L_{1,m}$         & Residue 1 & $L_{1,m}$         & Residue 2 & $L_{1,m}$         & Residue 3 & $L_{1,m}$         \\ \midrule
                        &          & 1           & 2.000000 & 2           & 2.000000 & 3           & 2.000000 \\
4                       & \textbf{1}.583333 & 5           & \textbf{1.7}33333 & 6           & \textbf{1}.814815 & 7           & \textbf{1}.863946 \\
8                       & \textbf{1}.687500 & 9           & \textbf{1.7}53086 & 10          & \textbf{1}.800000 & 11          & \textbf{1}.834711 \\
12                      & \textbf{1.7}22222 & 13          & \textbf{1.7}63314 & 14          & \textbf{1.79}5918 & 15          & \textbf{1}.822222 \\ \midrule
988                     & \textbf{1.79}0823 & 989         & \textbf{1.791}246 & 990         & \textbf{1.79166}8 & 991         & \textbf{1.79}2088 \\
992                     & \textbf{1.79}0827 & 993         & \textbf{1.791}248 & 994         & \textbf{1.79166}8 & 995         & \textbf{1.79}2086 \\
996                     & \textbf{1.79}0830 & 997         & \textbf{1.791}249 & 998         & \textbf{1.79166}8 & 999         & \textbf{1.79}2084 \\ \bottomrule
\end{tabular}
\end{table}

In summary, there are a few key differences between using integration and grid-summation for optimizing a polynomial over a continuous domain. First, our integration methods also give upper bounds $U_k$, while the grid-summation can only produce lower bounds. Second, $L_{k,m}$ converges to $L_k$ as $ m\rightarrow \infty$. Third, $m$ should be picked large enough to ensure that the optimal point in $\PP$ is close to a lattice point in $\PP \cap \frac{1}{m} \Z^d$ otherwise the sum $S(m)$ could only contain $f$ values that are far from the maximum $f$ value. The integration bounds have no analogous parameter. Fourth, and  most important for practical use, computing the sum $S(m)$ using generating functions required decomposing a polytope into unimodular cones, while the integration bounds do not need this. Decomposing to unimodular cones can significantly increase the number of cones that need processing.

\section{Decomposing a polynomial using Handelman decompositions}
\label{sec:handelman-generating-function}

From Theorem \ref{theorem:continuous-lkuk-bounds}, we see that it is necessary to exactly compute the integral $\int_{\PP} f(x)^k \d x$. Integrating a polynomial over a polytope by using efficient generating function methods has been done in  \cite{baldoni-berline-deloera-koeppe-vergne:integration,deloera:software-exact-integration-polynomials}. However, their  polynomial time algorithm for integrating $\int_{\PP} f(x)^k \d x$ relies on decomposing $h(x) := f(x)^k$ into a sum of powers of linear forms in the form $h(x) = \sum_i c_i \ll \ell_i, x\rr^{m_i}$ where $\ell_i \in \Q^d$, $c_i \in \Q$, and $m_i \in \N$. This was done by using the following formula on each monomial of $h$. If $x^{m}= x_1^{m_1}x_2^{m_2}\cdots x_d^{m_d}$, then

\begin{multline} 
x^{m}  = \frac{1}{|m|!} \sum_{0\leq p_i\leq m_i}(-1)^{|m|-(p_1+\cdots+p_d)}
  \binom{m_1}{p_1}\cdots \binom{m_d}{p_d}(p_1 x_1+\cdots+p_d x_d)^{|m|},
  \label{eq:decomp-powerlinform}
\end{multline}
where $|m| = m_1+\cdots+m_d$. Other decompositions into powers of linear forms could be done, for instance \cite{Carlini20125} describes a method with fewer terms but involves roots of unity. 

It is worth noting that Equation \eqref{eq:decomp-powerlinform} does \emph{not} yield an optimal decomposition. The problem of finding a decomposition with the smallest possible number of summands is known as the \emph{polynomial Waring problem} \cite{alexanderhirschowitz,brambillaottaviani,Carlini20125}. A key benefit of Equation \eqref{eq:decomp-powerlinform} is that it is explicit, is computable over $\Q$, and is sufficient for generating a polynomial-time algorithm on when the dimension is fixed \cite{baldoni-berline-deloera-koeppe-vergne:integration}. However, Equation \eqref{eq:decomp-powerlinform} has a major problem: it can produce many terms. For example, a monomial in $10$ variables and total degree $40$ can have more than $4$ million terms. To approximate optimization problems, we need to raise $f(x)$ to the power $k$, and so even if $f(x)$ has small degree or is sparse, $f(x)^k$ may not be, resulting in many terms from Equation \eqref{eq:decomp-powerlinform}. The large number of summands is our motivation for exploring an alternative decomposition using Handelman decompositions.

\subsection{Computing a Handelman decomposition}
Our goal is to use Handelman decompositions for integration and give empirical evidence that it can be superior to the power of linear form decomposition in that it can have fewer terms. Recall that Handelman's Theorem says  if $g_1,\dots,g_n\in\oR[x_1, \dots, x_d]$ are linear polynomials defining a polytope $\PP=\{x\in\R^d:g_1(x)\ge0,\dots,g_n(x)\ge0\}, $ then any polynomial $f\in\oR[x_1, \dots, x_d]$, strictly positive on $\PP$, can be written as the sum $f(x) = \sum_{\alpha\in\oN^n} c_\alpha g_1^{\alpha_1}\cdots g_n^{\alpha_n}$ for some nonnegative scalars $ c_{\alpha}$. We are interested in finding \emph{sparse Handelman decompositions}, with as few summands as possible.

If $f(x) = \sum_{\alpha\in\N^n} c_\alpha g_1^{\alpha_1}\cdots g_n^{\alpha_n}$, we say the \emph{degree} of a Handelman decomposition is $t := \max |\alpha|$, where the maximum is taken over all the exponent vectors $\alpha$ of $g_i(x)$ that appear in the decomposition. We call $c_\alpha g_1^{\alpha_1}\cdots g_n^{\alpha_n}$ a \emph{Handelman monomial}.  

Handelman decompositions have been used before to maximize a polynomial over $\PP$ (see \cite{laurentsurvey,monique2014,ParriloSturmfels2003,sankaranarayanan2013lyapunov}). 
Those methods require finding a Handelman decomposition of large Handelman degree $t$, which becomes more and more difficult as the degree of $t$ increases. Instead, our approach will prescribe a small $t$ value, and then search for a Handelman decomposition of degree $t$ for use in combination with Theorem \ref{theorem:continuous-lkuk-bounds}. The  key task we undertake can be summarized as follows: Given $t$, the goal is to find $c_\alpha \geq 0$ and a $s \in \R$ such that 
\[f(x) + s = \sum_{\alpha\in \Z_{\geq 0}^n \;:\; |\alpha|\le t} c_{\alpha} g^\alpha.\]
In fact, we add the unknown constant shift $s$ to $f(x)$ for three important reasons.

\begin{enumerate}
\item $-s$ will be a lower bound for $\fmin$. So even if $f$ is negative on $\PP$, $f(x)+s$ will be positive on $\PP$, allowing us to use Theorem \ref{theorem:continuous-lkuk-bounds}.
\item If the minimum value of $f$ is zero, $f$ might not have a Handelman decomposition for any $t$. By adding a shift, a decomposition is guaranteed to exist for some $t$. 

\item If the minimum value of $f$ is positive, but small, $f(x)$ might only have Handelman decompositions for large degree $t$. By adding a shift, we can find a Handelman decomposition of smaller size.
\end{enumerate} 

If we expand the right hand side of $f(x) + s = \sum_{|\alpha| \leq t} c_\alpha g^\alpha$ into monomials, then comparing the coefficients of monomials on both sides gives a linear system in the $c_\alpha$ and $s$. Hence we seek a solution to the \emph{linear program}
	\begin{align}
	\label{eq:sparse-lp-handelman}
	\min & \;s + \sum_{|\alpha| \leq t} c_\alpha\\ \nonumber
	& A_Hc_\alpha = b \\ 	\nonumber
	& a_0^Tc_\alpha -s = 0\\ \nonumber
	& s \text{ free}, c_\alpha \geq 0, 
	\end{align}
where the objective has been chosen so that $-s$ is close to $f_{\mathrm{min}}$ and to force a sparse Handelman decomposition of order $t$. It is common practice to use $\norm{\cdot}_1$ as a proxy for sparse solutions \cite{elad2010sparse}. Notice that $A_H$ has ${t+ n\choose n} = O(t^{n})$ columns if the number of facets $n$ is fixed, and $A_H$ has ${t + d \choose d} = O(t^d)$ rows if the dimension $d$ is fixed. 

Consider the example $f(x) = x^2 -x$ on $[-1,1] \subset \R$, which cannot have a Handelman decomposition as it takes negative values. 
We seek a solution to \[f(x) + s = c_{2,0} (x + 1)^2 + c_{1,1} (1 - x) (x + 1) + c_{0,2}(1 - x)^2 + c_{1,0} (x + 1) + c_{0,1}(1 - x).\]
Solving the linear program results in $c_{0,2} = 3/4$, $c_{2,0}=1/4$, and $s=1$.

In the remainder of this section, we will show that if $t$ is fixed to a certain value, then a large enough shift $s$ exists so that $f(x)+s$ has a degree $t$ Handelman decomposition. We first need some notation. Let $\Delta_k$ be the simplex $\{x \in \R^k \mid x_i \geq 0, \sum x_i = 1\}$, and for a polynomial in the form $h(x) = \sum_\alpha a_\alpha x^\alpha$, let $L(h) := \max\{ |a_\alpha| \frac{\alpha_1! \cdots \alpha_k!}{(\alpha_1+\cdots+\alpha_k)!} \}$. The next lemma will allow us to state a useful degree bound on a Handelman decomposition. 

\begin{lemma}[Theorem 1 in \cite{powers2001new}] 
\label{lemma:PowersPolya}
Suppose that $h(x) \in \R[x_1, \dots, x_d]$ is a homogeneous polynomial of degree $D$ and is strictly positive on $\Delta_d$. Let $\lambda$ be the minimum of $h$ on $\Delta_d$. If $N > \frac{D(D-1)}{2}\frac{L(h)}{\lambda} -D$, then $(x_1 +\cdots + x_d)^N h(x)$ has positive coefficients.
\end{lemma}

\begin{theorem}
\label{thm:large-s-given-t}
Let $f(x) \in \R[x_1, \dots, x_d]$ be any polynomial of degree $D$, and $\PP$ be a full dimensional polytope given by the linear polynomials $g_i(x)$ so that $\PP=\{x\in\R^d:g_1(x)\ge0,\dots,g_n(x)\ge0\}$.   Then there exists a large enough shift $s \in \R$ such that the polynomial $f(x)+s$ has a Handelman decomposition of degree $t=D(D-1) + 1$.
\end{theorem}

\begin{proof}
Our strategy is to follow the proof of Theorem 3 in \cite{powers2001new} and simplify the result of Theorem 5 in \cite{powers2001new}. However, we do not use their results 
directly.

By \cite{handelman1988} there exist a $c \in \R^n_{\geq 0}$ such that $\sum_i c_i g_i = 1$. Replace $c_ig_i(x)$ by $g_i(x)$ so that $\sum_i g_i(x) = 1$. Also, there exist constants $b_{j,i} \in \R$ so that for $j=1, \dots, d$, $x_j = \sum_{i=1}^n b_{j, i} g_i(x)$. Finding the $b_{j,i}$'s can be done by solving a linear system. Let $B$ be the real $d \times n$ matrix giving \[B \cdot (g_1(x), \dots, g_n(x))^T = (x_1, \dots, x_d).\] 

Define the map $\phi:\R[y_1, \dots, y_n] \rightarrow \R[x_1, \dots, x_d]$ by $y_i \mapsto g_i(x)$. If $f(x)$ has the form $f(x) = \sum_{|\alpha| \leq D} a_\alpha x^\alpha$, define the homogeneous polynomial $\widetilde{f}(y) \in \R[y_1, \dots, y_n]$ by \[\widetilde{f}(y) = \sum_{|\alpha| \leq D} a_\alpha (B\cdot (y_1, \dots, y_n)^T)^\alpha \cdot \left(\sum_{i=1}^n y_i\right)^{D - |\alpha|}.\]

If $s \in \R$ is large enough such that $f(B\cdot y)+s > 0$ for $y \in \Delta_n$, then with $\widetilde{h}(y) :=  \widetilde{f}(y) + s\cdot(y_1+\cdots +y_n)^D$, $\widetilde{h}(y) > 0$ for $y \in \Delta_n$. Let $\lambda(\widetilde{h})$ be the minimum of $\widetilde{h}$ on $y \in \Delta_n$. Applying Lemma \ref{lemma:PowersPolya} on $\widetilde{h}$ states that for $N > \frac{D(D-1)}{2}\frac{L(\widetilde{h})}{\lambda(\widetilde{h})} -D$, $(y_1 + \cdots + y_n)^N \cdot \widetilde{h}(y)$ has positive coefficients. Then we have that

\begin{equation}
\label{equ:poly-handelman}
(y_1 + \cdots +y_n)^N \cdot \widetilde{h}(y) = \sum_{|\alpha| = N+D} b_\alpha y_1^{\alpha_1} \cdots y_n^{\alpha_n},
\end{equation}
where $b_\alpha \geq 0$. Applying the map $\phi$ to both sides of Equation \ref{equ:poly-handelman} gives $f(x) + s$ decomposed as a Handelman polynomial of degree $N+D$. This theorem is finished once we show that $s$ can further be made large enough so that $\frac{L(\widetilde{h})}{\lambda(\widetilde{h})} \leq 2$. We do this by showing $\lim_{s \rightarrow \infty} \frac{L(\widetilde{h})}{\lambda(\widetilde{h})}  = 1$.

First note that $\lambda(\widetilde{h}) = \lambda(f)+s$ where $\lambda(f)$ is the minimum of $f(B\cdot y)$ on $y \in \Delta_n$. Notice that every monomial of $\widetilde{h}$ is in the form $(a'_\alpha + s \cdot \frac{D!}{\alpha_1!\cdots \alpha_n!}) y^\alpha$ for some $a'_\alpha \in \R$ and $|\alpha| = D$. Hence $L(\widetilde{h})$ is in the form $a'_\alpha \cdot \frac{\alpha_1!\cdots \alpha_n!}{D!}+ s$. Then $\lim_{s \rightarrow \infty} \frac{L(\widetilde{h})}{\lambda(\widetilde{h})}  = 1$.
\qed
\end{proof}

Note that $f(x)+s$ can have a Handelman decomposition of degree smaller than $D(D-1)+1$. In our experiments, we have always found decompositions of Handelman degree $D$.

\subsection{Benefits of Handelman decompositions}
\label{sec:handelman-example}
Next we compare between decomposing a random polynomial into a sum of powers of a linear form using Equation (\ref{eq:decomp-powerlinform}) and the Handelman method. We constructed a set of random polynomials in dimensions three, four, and five of total degree ranging from three to eight. For each dimension and degree pair, $20$ random polynomials where constructed. Hence we tested 360 polynomials in total. For each polynomial of degree $D$ in $d$ variables, $20\%$ of the possible $\binom{d+D}{d}$ monomials where randomly selected to have nonzero coefficients. The nonzero coefficients where uniformly selected from the set $[-10,10]\cap \Z -\{0\}$. For the Handelman decompositions, the polytope $\PP$ was picked to be the box $[-1, 1]^d$, $t$ was set to $D$, and the objective function enforces a sparse decomposition and minimizes $s$. 

Figure \ref{fig:percent-improvement-handelman-plf} illustrates how much better a Handelman decomposition can be over the power of linear form formula. The figure plots the average percent change between the number of terms in each method, and the bars reflect min and max percent reduction in the number of terms. For example, looking at the degree eight polynomials, the Handelman decompositions had about $40\%$ fewer terms than the power of linear forms formula among the 60 test polynomials ($20$ in dimension three, four, and five, each). Among the 60 test polynomials, the best case had about $50\%$ fewer terms while the worst case has about $30\%$ fewer terms. 

One interesting fact is that there are a few examples in degree three, four, and five where Handelman decompositions had a few more terms than the power of linear form decomposition (resulting in a negative percent improvement). However, Figure \ref{fig:raw-improvement-handelman-plf} reveals that the difference in the number of terms between both methods is small in these dimensions. Hence the benefit of a Handelman decomposition is less important for the low degree polynomials. Figure \ref{fig:percent-improvement-handelman-plf} does not show the degree 3 polynomials because their percent improvement ranged from $-50\%$ to $53\%$, meaning neither method consistently had fewer terms, but the number of terms from both methods is small. However, a Handelman decomposition also discovers a lower bound to shift $f(x)$ to make it nonnegative on $\PP$.

%\begin{figure}[h!]
%  \centering
%    \includegraphics[width=0.90\textwidth]{percentHandelmanVsPLF.png}
%      \caption{Percent improvement of the number of terms from a Handelman decomposition versus the number of terms from the power of linear form formula. Bars reflect min and max percent reduction in the number of terms.}
%    \label{fig:percent-improvement-handelman-plf}
%\end{figure}
%
%\begin{figure}[h!]
%  \centering
%    \includegraphics[width=0.90\textwidth]{totalHandelmanVsPLF.png}
%  \caption{Average number of terms between the Handelman decomposition and power of linear form decomposition.}
%   \label{fig:raw-improvement-handelman-plf}
%\end{figure}

\begin{figure}
        \centering
        \begin{subfigure}[b]{0.5\textwidth}
                \includegraphics[width=1.0\textwidth]{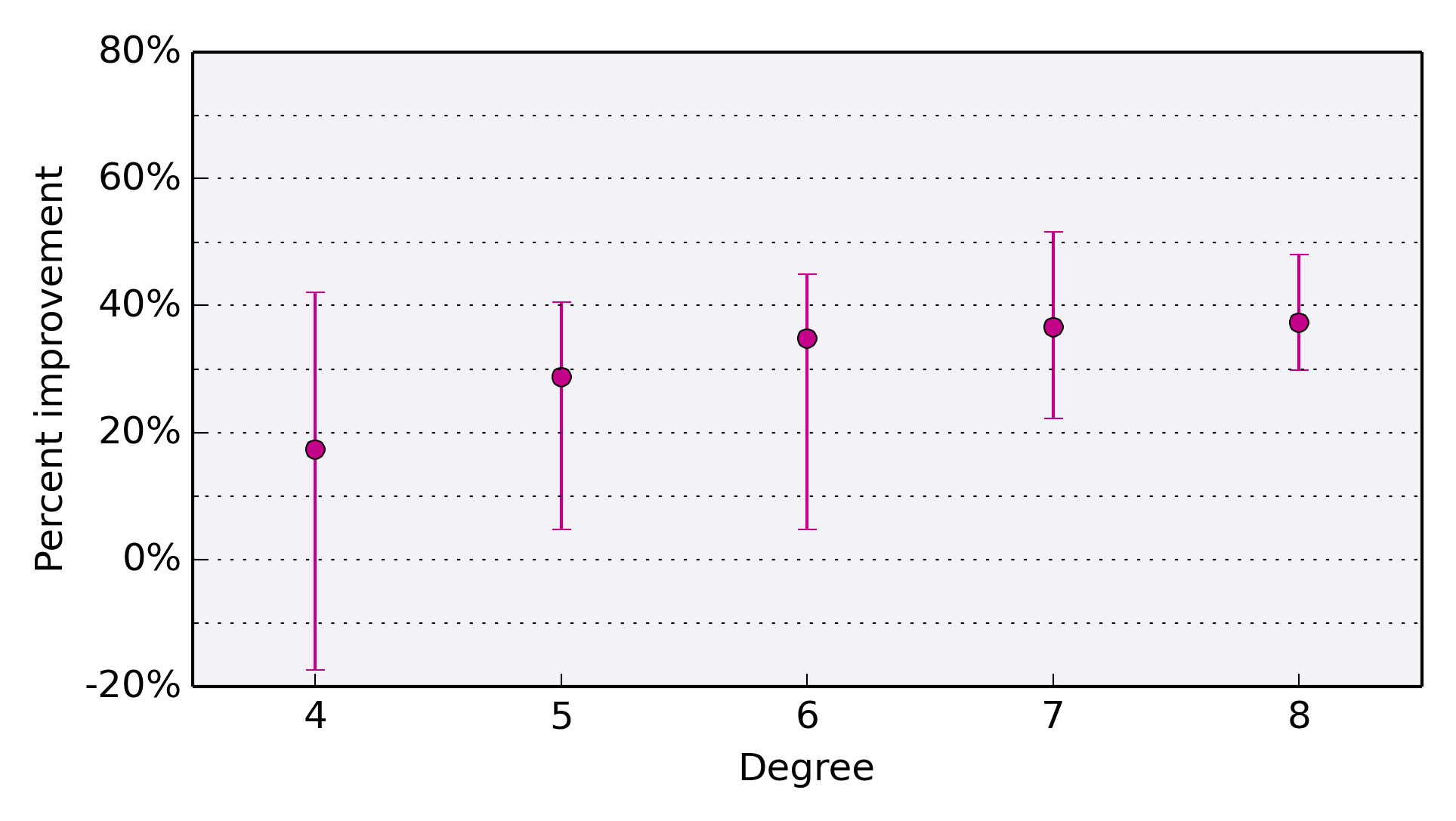}
                \caption{}
                \label{fig:percent-improvement-handelman-plf}
        \end{subfigure}%
        ~ %\hspace{1in} %add desired spacing between images, e. g. ~, \quad, \qquad, \hfill etc.
          %(or a blank line to force the subfigure onto a new line)
        \begin{subfigure}[b]{0.5\textwidth}
                \includegraphics[width=1.0\textwidth]{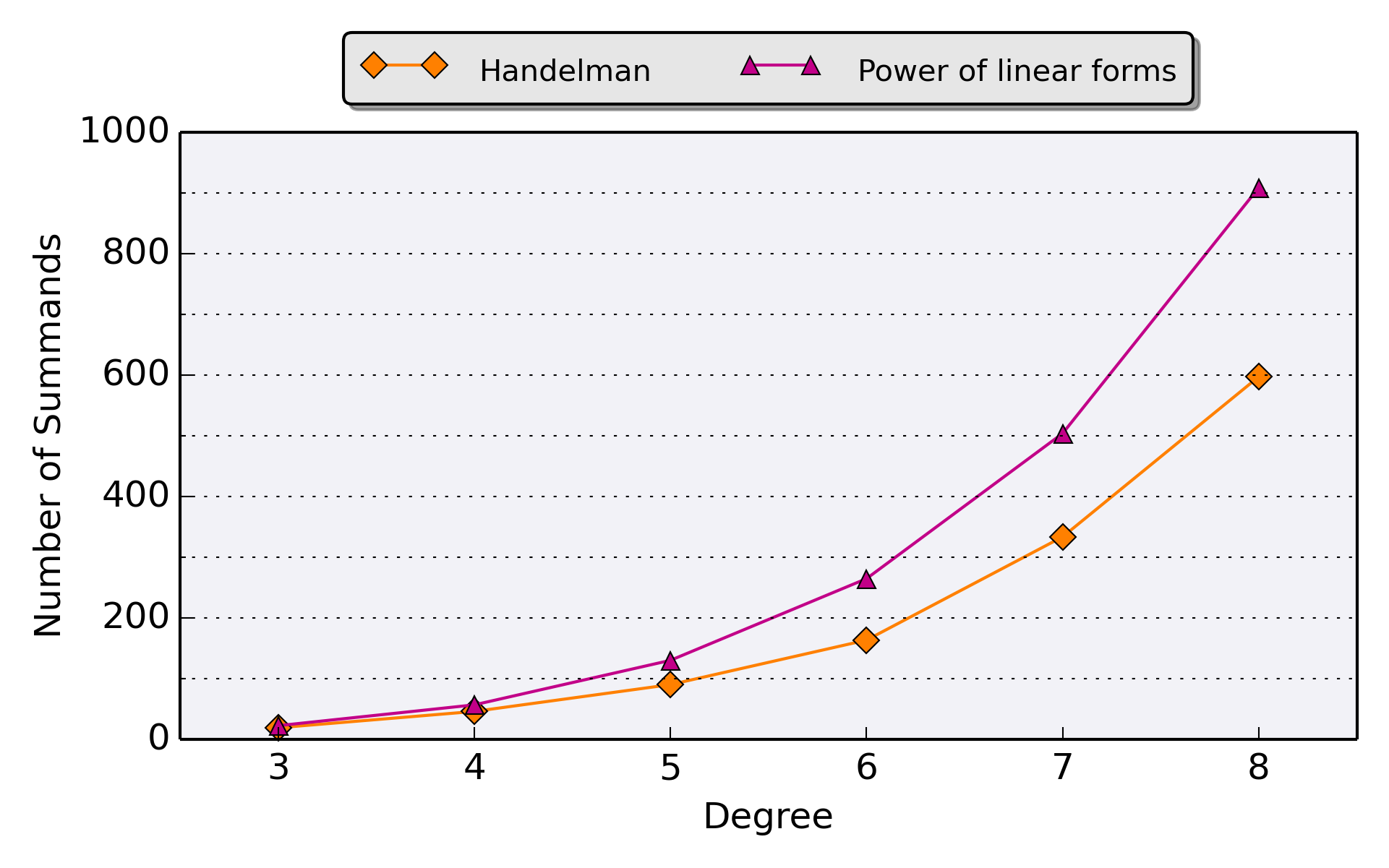}
                \caption{}
                \label{fig:raw-improvement-handelman-plf}
        \end{subfigure}
\caption{(\protect\subref{fig:percent-improvement-handelman-plf})Percent improvement of the number of terms from a Handelman decomposition versus the number of terms from the power of linear form formula. A positive percent means the Handelman decomposition had fewer terms.(\protect\subref{fig:raw-improvement-handelman-plf}) Average number of terms between a Handelman decomposition and power of linear form decomposition. Fewer terms is better.}
\label{fig:handelmanVsPLF}
\end{figure}

%\begin{figure}[h!]
%  \centering
%    \includegraphics[width=0.90\textwidth]{handelmanPercentImprovementDense.png}
%  \caption{Percent improvement of the average number of terms from a Handelman decomposition found by solving a linear program where the objective contains a sparsity term versus the number of nonzeros of a generic basic feasible solution. Bars reflect min and max percent improvements.}    
%    \label{fig:percent-improvement-handelman-sparsity}
%\end{figure}
%
%Next, we consider the effect of including a term in the objective function of Equation \eqref{eq:sparse-lp-handelman} to control the sparsity of solutions. We found that Figure \ref{fig:percent-improvement-handelman-sparsity} shows the average percent improvement between the number of linear forms found from our linear program and the number of nonzeros a generic basic feasible solution would have. For example, in degree six the Handelman linear program solutions contained about $30\%$ fewer terms than a generic basic feasible solution would have, and the sparsest solution contained $45\%$ fewer terms while the densest solution had $5\%$ fewer terms than a generic basic solution.  Every linear program is degenerate. This shows that it could be worthwhile to build an objective function that controls sparsity. Of course, any objective function that is not just $\min s$ can result in $-s$ being a poor lower bound for $\fmin$.

Next, we consider the effect of including a term in the objective function of Equation \eqref{eq:sparse-lp-handelman} to control the sparsity of solutions. We found that the average percent improvement between the number of linear forms found from our linear program and the number of nonzeros a generic basic feasible solution would have ranged from $20\%$ to $50\%$ for each degree class of polynomials. The densest solution among any test polynomial still had $5\%$ fewer terms than a generic basic solution. Every linear program we tested is degenerate. This shows that it could be worthwhile to build an objective function that controls sparsity. Of course, any objective function that is not just $\min s$ can result in $-s$ being a poor lower bound for $\fmin$. 

In terms of time to compute each decomposition, using Equation \eqref{eq:decomp-powerlinform} is really fast for the polynomials we considered. In fact, no example took longer than $0.1$ seconds to decompose into a sum of powers of linear forms. As Figure \ref{fig:raw-improvement-handelman-plf} shows, this is not surprising as about only a $1000$ terms was ever produced. However, as Table 1 in \cite{deloera:software-exact-integration-polynomials} shows, Equation \eqref{eq:decomp-powerlinform} can contain millions of terms already for just one monomial in dimension $10$ of degree $40$. For the Handelman decomposition, we used \texttt{MAPLE} 18 to build the linear systems and \texttt{SoPlex} 2.0.0 within \texttt{SCIP} 3.1.0 \cite{achterberg2009scip} for solving the linear systems . For polynomials of degree three to five, the average time to find a Handelman decomposition was $5$ seconds, while for the polynomials of degree six to eight, the average was $60$ seconds. The longest Handelman decomposition took $271$ seconds. 

One troubling fact with searching for a Handelman decomposition of order $t$ on a polytope $\PP$ with $n$ facets is that a linear program of size $\binom{t+d}{d} \times \binom{t+n}{n}$ needs to be solved. This linear program quickly becomes impractical for large dimensions or with complicated polytopes. One way to reduce this cost when $n$ is large is to contain $\PP$ within a larger simpler polytope like a box or simplex, or to decompose $\PP$ into simpler polytopes. Another idea is to use row and column generation methods for solving the linear program. These ideas and others will be considered in a future implementation of Handelman decompositions.

\section{A generating function for use with Handelman decompositions}
\label{sec:gen-fun-for-integration}

We expand on the work done in \cite{baldoni-berline-deloera-koeppe-vergne:integration,deloera:software-exact-integration-polynomials} and develop a new polynomial time method for exactly integrating a product of affine functions in polynomial time. This will allow the integration of terms from Handelman decompositions. The next few statements are a review of the types of generating functions that are key, and they are succinctly described in Chapter 8 of \cite{barvinokzurichbook}. We use $\ll \cdot, \cdot \rr$ to denote the standard dot product in $\R^d$.

\begin{proposition}%[Theorem 8.4 in \cite{barvinokzurichbook}]
[\cite{lawrence91-2}, and \cite{zbMATH00148667}]
\label{prop:eValuationInt}
There exists a unique valuation  $I(\;\cdot\;, \ell)$ which  associates  to every polyhedron
$\PP\subset \R^d$ a meromorphic function so that the following properties hold.

\begin{enumerate}
\item If $\ell$ is a linear form such that $e^{\ll \ell, x \rr}$ is integrable over $\PP$, then 
$I(\PP,\ell) = \int_\PP e^{\ll \ell, x \rr} \d x.$
\item For every point $s \in \R^d$, one has
$I(s+\PP, \ell) = e^{\ll \ell, s \rr}I(\PP,\ell).$
\item If $\PP$ contains a line, then $I(\PP, \ell) = 0$.
\end{enumerate}
\end{proposition}

The next lemma follows from the Brion--Lasserre--Lawrence--Varchenko decomposition theory of a polyhedron
into the supporting cones at its vertices \cite{beck-haase-sottile:theorema,Brion88,lasserre-volume1983}.

\begin{lemma}
\label{lemma:intTriangulation}
Let $\PP$ be a polyhedron with set of vertices $V(\PP)$. For each
vertex~$s$, let $C_s$ be the cone of feasible directions at vertex $s$. Then
\begin{equation*}
I(\PP,\ell)=\sum_{s\in V(\PP)}I(s+C_s,\ell).
\end{equation*}
\end{lemma}

The proof of the next two statements can be found in Lemma 8.3 in \cite{barvinokzurichbook}.
\begin{proposition}
\label{prop:eValuationIntCone}
  For a simplicial cone $C$ generated by rays $u_1,u_2,\dots u_d$ (with vertex $0$) and for any point $s$
\begin{equation*}
I(s+C,\ell) = \vol (\Pi_C)e^{\ll \ell, s \rr} \prod_{i=1}^d \frac{1}{-\ll \ell, u_i \rr}
\end{equation*}
where $\Pi_C$ is the fundamental parallelepiped generated by $C$ and $\vol (\Pi_C)$ is the absolute value of the determinant of the matrix containing the vectors $u_i$. These identities holds as a meromorphic function of~$\ell$ 
and pointwise for every $\ell$ such that $\langle \ell, u_i \rangle \neq 0$ for
all $u_i$.
\end{proposition}

\begin{lemma}
\label{lemma:integration-triangulation}
For any triangulation $\mathcal{D}_s$ of the feasible cone $C_s$ at each of the vertices $s$ of the polytope $\PP$ we have $I(\PP,\ell) = \sum_{s \in V(\PP)} \sum_{C \in \mathcal{D}_s} I(s+C,\ell)$.
\end{lemma}

We will repeatedly use the next lemma to multiply series in polynomial time in fixed dimension. The idea is to multiply each factor, one at a time, truncating after total degree $M$. 

\begin{lemma}[Lemma 4 in \cite{baldoni-berline-deloera-koeppe-vergne:integration}]
\label{lemma:poly-mult}
For every fixed $d$, there exists a polynomial time algorithm that takes as input a number $M$ in unary encoding, and a sequence of $k$ polynomials $h_1, \dots, h_k \in \Q[x]$ in $d$ variables of total degree at most $M$ in dense monomial representation, and outputs the product $h_1\cdots h_k$ truncated at total degree $M$ in dense monomial representation. 
\end{lemma}

The algorithm in Lemma \ref{lemma:poly-mult} performs $O(kM^{2d})$ elementary rational operations, see \cite{baldoni-berline-deloera-koeppe-vergne:integration}.  The next proposition contains our new polynomial time algorithm for integrating a product of affine functions. 

\begin{proposition}
\label{prop:continuous-affnie-products-linear-forms}
When the dimension $d$ and number of factors $n$ is fixed, the value of 
\[\int_{\PP} \frac{(\langle \ell_1, x\rangle + r_1)^{m_1} \cdots \langle (\ell_n, x\rangle + r_n)^{m_n}}{m_1!\cdots m_n!} \d x\]
can be computed in polynomial time in $M:=\sum_{i=1}^nm_i$ and the binary encoding size of $\PP$, $\ell_i$ and $r_i$ for $1 \leq i \leq n$. 
\end{proposition}

\begin{proof}
We will compute the polynomial
\begin{equation}
\label{eq:intAffineProdAnswer}
\sum_{p_1+\cdots+p_n \leq M} \left( \int_{\PP} \frac{(\langle \ell_1, x\rangle + r_1)^{p_1} \cdots (\langle \ell_n, x\rangle + r_n)^{p_n}}{p_1!\cdots p_n!}  \d x \right) t_1^{p_1}\cdots t_n^{p_n}. 
\end{equation}
in dense representation in polynomial time in $M$ and the input data, which contains the desired value as a coefficient of a polynomial in $t_1, \dots, t_n$. This is useful when integrating a polynomial in a Handelman form.

Using Propositions \ref{prop:eValuationInt} and \ref{prop:eValuationIntCone} and Lemmas \ref{lemma:intTriangulation} and \ref{lemma:integration-triangulation}, we start with the exponential integral in the indeterminate $\ell$:
\begin{equation}
\label{eq:intOverE}
\int_{\PP} e^{\langle \ell, x\rangle} \d x
=
\sum_{s \in V(\PP)}\sum_{C \in \mathcal{D}_s} \mathrm{vol}(\Pi_C) e^{\ll \ell, s\rr}  \prod_{i=1}^d \frac{1}{-\ll \ell, u_i\rr}.
\end{equation}
Note that because the dimension is fixed, the number of vertices and the number of simplicial cones at each feasible cone of $\PP$ is polynomial in the input size of $\PP$, see Chapter 8---and in particular the Upper Bound Theorem---in  \cite{DRStriangbook}.

First pick $\ell_{n+1} \in \Q^d$ so that $\ll \ell_{n+1}, u \rr \neq 0$ for every ray $u$ in the simplicial cones $\mathcal{D}_s$ at each vertex $s$. The set of points $\ell_{n+1}$ that fail this condition have measure zero, so $\ell_{n+1}$ can be picked randomly. Next, replace $\ell$ with $\ell_1t_1 + \cdots + \ell_{n+1}t_{n+1}$ in Equation \eqref{eq:intOverE}. To simplify notation, let $\ve t = (t_1, \dots, t_n)$, $\ve a_s = ( \ll \ell_1, s \rr, \dots,  \ll \ell_{n}, s \rr)$, $\ve r:= (r_1, \dots,  r_n)$, $\ve b_i := (\ll \ell_1, u_i \rr,  \dots, \ll \ell_{n}, u_i \rr)$, $\beta_i := \ll \ell_{n+1}, u_i \rr$, and $\boldsymbol{\ell_x} = (\ll \ell_1, x\rr, \dots, \ll \ell_{n}, x\rr)$. Then Equation \eqref{eq:intOverE} becomes
\begin{equation}
\label{eq:intOverE-sub}
\int_{\PP} e^{ \ll \boldsymbol{\ell_x}, \ve t\rr  + \ll \ve  r , \ve t \rr + \ll \ell_{n+1}, x\rr t_{n+1}}\d x
=
\sum_{s \in V(\PP)}\sum_{C \in \mathcal{D}_s} \mathrm{vol}(\Pi_C)   \prod_{i=1}^d \frac{e^{\ll \ve a_s, \ve t \rr}e^{\ll \ve r, \ve t\rr}e^{\ll \ell_{n+1}, s \rr t_{n+1}}}{-\ll \ve b_i, \ve t \rr - \beta_it_{n+1}}.
\end{equation}
If the $e^{(\ll \ell_i, x\rr + r_i) t_i}$ in Equation \eqref{eq:intOverE-sub} are expanded according to $e^z = \sum \frac{z^k}{k!}$ and multiplied together, then the left hand side of Equation \eqref{eq:intOverE-sub} is 
\begin{equation}
\label{eq:intAffineProdInfSum}
\sum_{p_i \in \Z_{\geq 0}} \left( \int_{\PP} \frac{(\langle \ell_1, x\rangle + r_1)^{p_1} \cdots (\langle \ell_n, x\rangle + r_n)^{p_n} \cdot \ll \ell_{n_1}, x \rr^{p_{n+1}}}{p_1!\cdots p_{n+1}!}  \d x \right) t_1^{p_1}\cdots t_{n+1}^{p_{n+1}}. 
\end{equation}
To compute Equation \eqref{eq:intAffineProdAnswer} it is sufficient to compute the series expansion of each summand in the right hand side of Equation \eqref{eq:intOverE-sub} in $t_1, \dots, t_n$ up to total degree $M$ where the power of $t_{n+1}$ is zero, and do this in polynomial time in $M$. Let $h_i(t) = \sum_{k=0}^M (\ll \ell_i, s\rr + r_i)^k t_i^k /k!$ for $1 \leq i \leq n$ and note that $h_i(t)$ is the truncated series expansion of $e^{(\ll \ell_i, s\rr + r_i)t_i}$. For $1 \leq i \leq d$ let

\[h_{n + i}(t) = \sum_{k=0}^M (-1)^k (\ll \ve b_i, \ve t \rr)^k (\beta_{ij} t_{n+1})^{-1-k}\] using a dense polynomial representation. Note that $h_{n+i}$ is the truncated series from the generalized binomial  theorem applied to $\frac{1}{-\ll \ve b_i, \ve t \rr - \beta_it_{n+1}}$. Each $h_{n+i}$ for $1 \leq i \leq d$ is a polynomial in $t_1, \dots, t_n$ of degree $M$ with negative powers on $t_{n+1}$. Applying Lemma \ref{lemma:poly-mult} to $H_1 := \prod_{i=1}^{n+d} h_i$ results in the series expansion of 

\[\prod_{i=1}^d \frac{e^{\ll \ve a_s, \ve t \rr}e^{\ll \ve r, \ve t\rr}}{-\ll \ve b_i, \ve t \rr - \beta_it_{n+1}}.
\]
up to total degree $M$ in $t_1, \dots, t_{n}$ and where the power of $t_{n+1}$ at most ranges from $-d(M+1)$ to $-d$. 

Let $h_{n+d+1}(t)$ be the series expansion of $e^{\ll \ell_{n+1}, s \rr t_{n+1}}$ up to degree $d(M+1)$ in $t_{n+1}$. Next, use Lemma \ref{lemma:poly-mult} one last time while treating $t_{n+1}$ as a coefficient and truncating at total degree $M$ to compute $H_2 := H_1 h_{n+d+1}$. Any term where the power of $t_{n+1}$ is not zero can be dropped because $I(\PP, \ell_1t_1 + \cdots + \ell_{n+1}t_{n+1})$ is holomorphic in $t_{n+1}$. Repeating this calculation for every summand in the right hand side of Equation \eqref{eq:intOverE-sub} and adding the results  produces Equation \eqref{eq:intAffineProdAnswer}.
\qed
\end{proof}

%%%%%%%%%%%%%%%%%%%%%%%%%%%%%%%%%%%%%%%%%%%%%%%%%
%%%%%%%%%%%%%%%%%%%%%%%%%%%%%%%%%%%%%%%%%%%%%%%%%
%%%%%%%%%%%%%%%%%%%%%%%%%%%%%%%%%%%%%%%%%%%%%%%%%

\section{Complexity guarantees for Handelman and generating functions} \label{proof-thm-last}

Algorithm \ref{alg:knorm-handelman} below gives the final connection between optimizing a polynomial, computing integrals, and a Handelman decomposition. 

\begin{algorithm}
\caption{Computing $\int_\PP (f(x)+s)^k\d x$ via Handelman}
\label{alg:knorm-handelman}
\begin{flushleft}
Input: A polynomial $f(x) \in \Q[x_1, \dots, x_d]$ of degree $D$, polytope $\PP$, and $k$.

Output: $s$ such that $f(x)+s \geq 0$ on $\PP$, and $\int_\PP (f(x)+s)^k \d x$.
\end{flushleft}

\begin{enumerate}
\item Let $t \leftarrow D(D-1)+1$
\item  Find $s \in \R$ and $c_\alpha \geq 0$ such that $f(x)+s = \sum_{|\alpha| \leq t} c_\alpha g^\alpha$ by solving the linear program in Equation \eqref{eq:sparse-lp-handelman}
\item Expand $h(x) := (\sum_{|\alpha| \leq t} c_\alpha g^\alpha)^k$ in terms of $g$ to get a new sum of products of affine functions
\item Integrate each Handelman monomial in $h(x)$ by the method in Proposition \ref{prop:continuous-affnie-products-linear-forms} and sum the results
\end{enumerate}
\end{algorithm}

The proof of the following proposition gives a proof of part (1) of Theorem \ref{theorem:handelman-good-s-on-simplex}.

\begin{proposition}
Fix the dimension $d$ and the number of facets $n$ in $\PP$.  Algorithm \ref{alg:knorm-handelman} runs in time polynomial in the unary encoding size of $k$ and $D$, and the binary encoding  size of $f$ and $\PP$.
\end{proposition}

\begin{proof}
By Theorem \ref{thm:large-s-given-t}, $f(x)+s$ has a degree $t=D(D-1)+1$ Handelman decomposition for some $s$. Then because $t$ is $O(D^2)$, the linear program that is solved has a matrix with dimension at most $O(D^d) \times O(D^{2n})$. Expanding $(\sum_{|\alpha| \leq t} c_\alpha g^\alpha)^k$ has at most ${ kt + n \choose n} = O((kD^2)^n)$ Handelman terms. Finally, each term can be integrated in time polynomial in $kt$ by Proposition \ref{prop:continuous-affnie-products-linear-forms}.  \qed
\end{proof}

Notice that Steps (2) and (3) in Algorithm \ref{alg:knorm-handelman} can be swapped, resulting in the same time complexity. For example, $h(x) := (f(x) +s)^k$ could first be expanded, and then a Handelman decomposition could be found for $h(x)$.

This method is quite good in practice, but  one wishes to have control on the values of the constant $s$ that are compatible with  Theorem \ref{theorem:continuous-lkuk-bounds} in the sense
that the maximum for $f$ is approximated well. When $\PP$ is the standard simplex the size of $s$ can be better controlled. This guarantee can be provided using the following result of Laurent and Sun.

\begin{lemma}[See page 4 in \cite{monique2014}]
\label{lemma:s-bound-for-simplex}
Let $q(x)$ be a polynomial of degree $D$ and let $s = \min \lambda$ such that $\lambda - q(x)$ has a Handelman decomposition of degree $t \geq D$ on the standard simplex $\Delta_d = \{x \in \R^d \mid x_i \geq 0, \sum x_i = 1\}$. Then $s - q_{\text{\rm max}} \leq D^D \binom{2D-1}{D} \frac{\binom{D}{2}}{t - \binom{D}{2}} (q_{\text{\rm max}} - q_{\text{\rm min}})$.
\end{lemma}

%\begin{lemma}
%\label{lemma:polynomial-constant-on-simplex}
%Let $h(x)$ be a degree $D$ polynomial on a full-dimensional simplex $\Delta \subset \R^d$. If $h(x)$ is constant on $\Delta$, then $h(x)$ is identically a constant polynomial in $\R^d$.
%\end{lemma}
%
%\begin{proof}
%Without loss of generality, assume $h(x)$ is zero on $\Delta$. Then because $\Delta$ is full dimensional, there is a $m \in \Z$ such that there is a subset of $\frac{1}{m}\Z^d \subset \Delta$ in the form of a cube:
%\[S_1:=\left\{ \frac{a_1}{m}, \frac{a_1+1}{m}, \dots, \frac{a_1+D}{m}\right\} \times \cdots \times \left\{\frac{a_d}{m}, \frac{a_d+1}{m}, \dots, \frac{a_d+D}{m}\right\}.\]
%Because $h(x)$ is zero on $S_1$, the polynomial $h'(x_1, \dots, x_d) := h(\frac{1}{m}(x_1 + a_1), \dots, \frac{1}{m}(x_1 + a_d))$ is zero on $\{0, \dots, D\}^d$. By Lemma 12 in \cite{deloera-hemmecke-koeppe-weismantel:mixedintpoly-fixeddim}, $h'(x) = 0$  on $\R^d$, and thus so too is $h(x)$. 
%\qed
%\end{proof}
%

\begin{proof}\textit{of Theorem \ref{theorem:handelman-good-s-on-simplex}, part 2:}

What remains to be shown is the statement when $\PP$ is a simplex. Let $\Delta$ be a full dimensional simplex in $\R^d$ given by the convex hull of its vertices: $\Delta = \Conv(u_1, \dots, u_{d+1})$. Let $H$ be the homogenization map given by $(x_1,\dots,x_d) \mapsto (x_1,\dots,x_d,1)$, and $\Pi$ be the corresponding projection given by $(x_1,\dots,x_d, x_{d+1}) \mapsto (x_1,\dots,x_d)$. Let $U \in \R^{d+1 \times d+1}$ be the invertible matrix whose columns are $Hu_i$ for $i=1,\dots, d+1$, then $U$ is a linear map from $\Delta_{d+1}$ to $H\Delta$. We now apply Lemma \ref{lemma:s-bound-for-simplex} with $q(y) = -f(\Pi Uy)$ for $y \in \Delta_{d+1}$.

First note that \[q_{\text{\rm min}} = \min_{y \in \Delta_{d+1}} q(y) = \min_{y \in \Delta_{d+1}} -f(\Pi Uy) = \min_{x \in \Delta} -f(x) = -\fmax,\]
and similarly, $q_{\text{\rm max}} = -\fmin$. Let $s' = \min \lambda$ such that $f(x)+\lambda$ has a Handelman decomposition of degree $t=D(D-1)+1$ on $\Delta$. Unlike Equation \eqref{eq:sparse-lp-handelman}, there is no sparsity term in the objective. Let $s = \min \lambda$ such that $f(\Pi Uy)+\lambda$ has a Handelman decomposition of degree $t=D(D-1)+1$ on $\Delta_{d+1}$. We claim $s' = s$. First we show $s' \leq s.$ There exist $c_\alpha \geq 0$ such that 
\begin{equation}
\label{eq:f-handelman-on-std-simplex}
f(\Pi Uy)+s = \sum_{|\alpha| \leq t} c_\alpha \prod_{i=1}^{d+1}(y_i)^{\alpha_i}\left(1 -\sum y_i\right)^{\alpha_{d+1}}\left(\sum y_i - 1\right)^{\alpha_{d+2}}, 
\end{equation}
where the factors associated with $\alpha_{d+1}$ and $\alpha_{d+1}$ come from writing the equality constraint in $\Delta_{d+1}$ as two inequality constraints. We will rewrite Equation \eqref{eq:f-handelman-on-std-simplex} as a Handelman decomposition on $\Delta$. By \cite{fukuda1996double} applied to a simple cone, the rows of $U^{-1}Hx = U^{-1}\begin{pmatrix}
x \\ 1
\end{pmatrix} \geq 0$ is the inequality description of $\Delta$ (up to a scaling). Let $y = U^{-1}Hx$. Write  $u \in \Delta$ as a convex combination of the $u_i$: $u = \sum_{i=1}^{d+1} \beta_i u_i$ where $b_i \geq 0$ and $\sum b_i = 1$. Then if $\textbf{1}$ is the all-ones vector in $\R^{d+1}$, $\sum y_i = \textbf{1}^Ty = \textbf{1}^TU^{-1}(\sum \beta_i Hu_i) = 1$. Therefore $\sum y_i = 1$ under the map $y = U^{-1}Hx$ for all $x \in \Delta$. Because $\Delta$ is full-dimensional, $\sum y_i = 1$ for every $x\in \R^d$ with $y = U^{-1}Hx$. Applying $y = U^{-1}Hx$ to both sides of Equation \eqref{eq:f-handelman-on-std-simplex} results in $f(x)+s$ written as a Handelman polynomial of degree $t$ on $\Delta$ and so $s' \leq s$. A similar argument shows $s \leq s'.$ Therefore, let $s$ be computed by solving the  linear program $s = \min \lambda$ such that $f(x)+\lambda$ has a degree $t=D(D-1)+1$ Handelman decomposition on $\Delta$, similar to Equation \eqref{eq:sparse-lp-handelman}. This is done in polynomial time in the unary encoding size of $D$ and binary encoding size of $f(x)$ and $\PP$ when $d$ is fixed. The dimensions of the linear system is $O(t^d)\times O(t^{d+1})$. 

Next, applying Lemma \ref{lemma:s-bound-for-simplex} with $t=D(D-1)+1$ and $g(y) = -f(\Pi Uy)$ on $\Delta_{d+1}$ shows that 
\begin{equation}
\label{eq:s-bound-on-simplex}
 s + \fmin \leq D^D \binom{2D-1}{D} \frac{\binom{D}{2}}{t - \binom{D}{2}} (\fmax - \fmin) \leq D^D \binom{2D-1}{D} (\fmax - \fmin).
\end{equation}
Define $c:= D^D \binom{2D-1}{D}$ and $\epsilon' := \epsilon \frac{1}{2c}$. Then applying Theorem \ref{theorem:continuous-lkuk-bounds} with $\epsilon'$ to $f(x)+s$ together with Equation \eqref{eq:s-bound-on-simplex} shows that
\begin{align*}
U_k - L_k &\leq \epsilon' (\fmax + s) \\
& \leq \epsilon' (\fmax -\fmin + c (\fmax -\fmin)) \\
& = \epsilon' (\fmax -\fmin) + \epsilon' (c\cdot (\fmax -\fmin)) \\
& \leq 2\epsilon' c (\fmax -\fmin) \\
& = \epsilon (\fmax -\fmin).
\end{align*}
\qed
\end{proof} 

In conclusion, Algorithm \ref{alg:knorm-handelman} is a simple practical approximation algorithm for optimizing a polynomial over a polytope, and we hope the theoretical guarantees of Theorem \ref{theorem:handelman-good-s-on-simplex} 
 will be extended further in the future.

%%%%%%%%%%%%%%%%%%%%%%%%%%%%%%%%%%%%%%%%%%%%%%%%%%%%%%%%%%%%%%%
%%%%%%%%%%%%%%%%%%%%%%%%%%%%%%%%%%%%%%%%%%%%%%%%%%%%%%%%%%%%%%%
%%%%%%%%%%%%%%%%%%%%%%%%%%%%%%%%%%%%%%%%%%%%%%%%%%%%%%%%%%%%%%%
%%%%%%%%%%%%%%%%%%%%%%%%%%%%%%%%%%%%%%%%%%%%%%%%%%%%%%%%%%%%%%%

%%%%%%%%%%%%%%%%%%%%%%%%%%%%%%%%%%%%%%%%%%%%%%%%%%%%%%%%%%%%%%%
%%%%%%%%%%%%%%%%%%%%%%%%%%%%%%%%%%%%%%%%%%%%%%%%%%%%%%%%%%%%%%%
%%%%%%%%%%%%%%%%%%%%%%%%%%%%%%%%%%%%%%%%%%%%%%%%%%%%%%%%%%%%%%%
%%%%%%%%%%%%%%%%%%%%%%%%%%%%%%%%%%%%%%%%%%%%%%%%%%%%%%%%%%%%%%%

%\section{Conclusions}
%\label{sec:conclusions}
%
%We have established lower and upper bounds for the maximum of a polynomial in terms of integrals of powers of $f$. 
%We explored how and when  Handelman's  decomposition of a polynomial can be helpful and efficiently computed and we developed a new efficient way to integrate a product of affine functions.  
%

\section{Acknowledgments}
The authors are grateful to the excellent and very detailed comments of the referees. Their comments and care
made a major contribution to improve the quality of this paper and outlined many new ideas and questions.
Jes\'us A. De Loera was supported by the NSA. Brandon Dutra was supported by NSF DGE-1148897. 

\bibliography{biblioCopy}{}
\bibliographystyle{spmpsci}

\end{document}